\documentclass[11pt]{amsart} 
\usepackage[lmargin=1in,rmargin=1in,tmargin=1in,bmargin=1in]{geometry}
\usepackage[ps,all,arc,rotate]{xy}
\usepackage{graphicx, float, epstopdf}
\usepackage{bbm}
\usepackage{hyperref, color}
\usepackage{centernot}
\usepackage{fancyhdr}
\usepackage{multirow}
\usepackage[utf8]{inputenc}
\usepackage{amsfonts,amssymb,amsmath,amsthm,mathrsfs}
\usepackage{graphics, setspace}
\usepackage{braket}
\usepackage{mathtools}
\usepackage{tikz}  
\usepackage{upgreek}

\numberwithin{equation}{section}
\numberwithin{figure}{section}
\allowdisplaybreaks[4]         %allows breaking multiline environments in amsmath commands
 
\newtheorem{lemma}{Lemma}[section]
\newtheorem{theorem}{Theorem}[section]
\newtheorem{proposition}{Proposition}[section]

\newtheorem{corollary}[lemma]{Corollary}
\theoremstyle{definition}
\newtheorem{definition}{Definition}[section]
\newtheorem{remark}{Remark}[section]
\usepackage{booktabs}
\usepackage{array}
\usepackage{vcell}
\newcommand{\R}{\mathbb{R}}

\newcommand{\N}{\mathbb{N}}

\newcommand{\e}{\operatorname{e}}

\newcommand{\modu}{\operatorname{mod}}

\newcommand{\norm}[1]{\left\lVert#1\right\rVert}

\newcommand{\mycomment}[1]{}

\begin{document}
\title[Exponential sums over M\"{o}bius functions with partitions]{Exponential sums over M\"{o}bius convolutions with applications to partitions}
%%%%%%%%%%%%%%%
 \author[D.~Basak]{Debmalya Basak}
\address{Debmalya Basak: Department of Mathematics, University of Illinois, 1409 West Green Street, Urbana, IL 61801, USA}
\email{dbasak2@illinois.edu}
\author[N.~Robles]{Nicolas Robles}
\address{Nicolas Robles: RAND Corporation, Engineering and Applied Sciences, 1200 S Hayes, Arlington, VA 22202, USA}
\email{nrobles@rand.org}
\author[A.~Zaharescu]{Alexandru Zaharescu}
\address{Alexandru Zaharescu: Department of Mathematics, University of Illinois, 1409 West Green
Street, Urbana, IL 61801, USA; \textnormal{and} Institute of Mathematics of the Romanian Academy, P.O. BOX 1-764, Bucharest, Ro-70700, Romania}
\email{zaharesc@illinois.edu}
%%%%%%%%%%%%%%%
\maketitle
%%%%%%%%%%%%%%%%%%%%%%%%%%%%%%%%%%%%%%%%%%%%%%%%%%%%%%%%%%%%%%%%%%%%%%%%%%%%%%%%%%%%%%%%%%%%%%%%%
\begin{abstract}
We consider partitions $p_{w}(n)$ of a positive integer $n$ arising from the generating functions
\[\sum_{n=1}^\infty p_{w}(n) z^n =  \prod_{m \in \mathbb{N}} (1-z^m)^{-w(m)},
\]
where the weights $w(m)$ are M\"{o}bius convolutions. We establish an upper bound for $p_w(n)$ and, as a consequence, we obtain an asymptotic formula involving the number of odd and even partitions emerging from the weights. In order to achieve the desired bounds on the minor arcs resulting from the Hardy-Littlewood circle method, we establish bounds on exponential sums twisted by M\"{o}bius convolutions. Lastly, we provide an explicit formula relating the contributions from the major arcs with a sum over the zeros of the Riemann zeta-function. 
\end{abstract}
%%%%%%%%%%%%%%%%%%%%%%%%%%%%%%%%%%%%%%%%%%%%%%%%%%%%%%%%%%%%%%%%%%%%%%%%%%%%%%%%%%%%%%%%%%%%%%%%%
\section{Introduction} \label{sec:introduction}
\subsection{Motivation}\label{Motivation}
A partition of a positive integer $n$ is a non-decreasing sequence of positive integers whose sum is equal to $n$. The partition function $p(n)$, which represents the number of partitions of $n$, was first studied by Hardy and Ramanujan \cite{hardyramanujan} in 1918 by applying the Hardy-Littlewood circle method. Their result states that
\begin{align*}
    p(n) \sim \frac{1}{4n\sqrt{3}} \exp \bigg(\pi\sqrt{\frac{2n}{3}}\bigg) \quad \textnormal{as} \quad n \to \infty.
\end{align*}
We consider a question about partitions of a positive integer that has a strong arithmetic flavor. To be precise, we look at the parts that appear in the partition and more specifically, we consider the prime factorization of the parts. Let $B$ (blue) and $R$ (red) denote the set of squarefree numbers that have an even, and respectively, an odd number of prime factors. For the convenience of the reader, we use colors to show parity of these partitions. Thus, we write
\begin{align*}
B &= \{ \textcolor{blue}{1, 6, 10, 14, 15, 21, 22, 26,\dots} \}, \nonumber \\
R &= \{ \textcolor{red}{2, 3, 5, 7, 11, 13, 17, 19, 23, 29, 30, \dots}  \}.
\end{align*}
We note in passing that the sets $B$ and $R$ are asymptotically just as numerous. Actually, it is expected that as $x$ tends to infinity, the number of elements of $B$ less than $x$ differs from the number of elements of $R$ less than $x$ by at most $O_{\epsilon} (x^{1/2 + \epsilon})$ for any $\epsilon>0$. As is well-known, this is equivalent to the Riemann hypothesis.

\begin{definition}
We call a partition of a positive integer $n$ \textit{admissible} if each part is either blue or red and each red part appears at most once. We say that an admissible partition is even (respectively odd) if the number of red parts used in the partitions is even (and respectively odd).
\end{definition}
 For example, there are exactly 15 admissible partitions in the case when $n=10$. The partition
\begin{align*}
    10 = \textcolor{blue}{6} + \textcolor{red}{2} + \textcolor{red}{2} 
\end{align*}
is not admissible because the red 2 appears twice. We see in Table \ref{Tab:Tcr} that there are $8$ even admissible partitions and $7$ odd admissible partitions in this situation. This naturally leads to the following question.

\begin{table}[ht]
\begin{tabular}{@{}c@{}|l|c|r|}
\toprule
& Case for $n=10 \textrm{ }$ 
& Admissible partitions
\\ \midrule
\savecellbox{\rule{0pt}{100pt}}
& \savecellbox{Odd admissible partitions}
& \savecellbox{$\begin{aligned}[t] % <-- note the 't' ("top") placement specifier
10 &= \textcolor{red}{7} + \textcolor{blue}{1} + \textcolor{blue}{1} + \textcolor{blue}{1}\\
10 &= \textcolor{blue}{6} + \textcolor{red}{3} + \textcolor{blue}{1}    \\ 
10 &= \textcolor{blue}{6} + \textcolor{red}{2} + \textcolor{blue}{1} + \textcolor{blue}{1}   \\
10 &= \textcolor{red}{5} + \textcolor{red}{3} + \textcolor{red}{2}  \\
10 &= \textcolor{red}{5} + \textcolor{blue}{1} + \textcolor{blue}{1} + \textcolor{blue}{1} + \textcolor{blue}{1} + \textcolor{blue}{1}  \\
10 &= \textcolor{red}{3} + \textcolor{blue}{1} + \textcolor{blue}{1} + \textcolor{blue}{1} + \textcolor{blue}{1} + \textcolor{blue}{1} + \textcolor{blue}{1} + \textcolor{blue}{1}  \\
10 &= \textcolor{red}{2} + \textcolor{blue}{1} + \textcolor{blue}{1} + \textcolor{blue}{1} + \textcolor{blue}{1} + \textcolor{blue}{1} + \textcolor{blue}{1} + \textcolor{blue}{1} + \textcolor{blue}{1}
\end{aligned} $}
\\[-\rowheight]
\printcellmiddle
& \printcellmiddle
& \printcellmiddle
%& \printcellmiddle
%\\ \bottomrule
\\ \midrule
\savecellbox{\rule{0pt}{115pt}}
& \savecellbox{Even admissible partitions}
& \savecellbox{$\begin{aligned}[t] % <-- note the 't' ("top") placement specifier
\quad \; \; 10 &= \textcolor{blue}{1} + \textcolor{blue}{1} + \textcolor{blue}{1} + \textcolor{blue}{1} + \textcolor{blue}{1} + \textcolor{blue}{1} + \textcolor{blue}{1} + \textcolor{blue}{1} + \textcolor{blue}{1} + \textcolor{blue}{1}  \\ 
10 &= \textcolor{red}{3} + \textcolor{red}{2} + \textcolor{blue}{1} + \textcolor{blue}{1} + \textcolor{blue}{1} + \textcolor{blue}{1} + \textcolor{blue}{1}   \\ 
10 &= \textcolor{red}{5} + \textcolor{red}{2} + \textcolor{blue}{1} + \textcolor{blue}{1} + \textcolor{blue}{1}  \\
10 &= \textcolor{red}{5} + \textcolor{red}{3} + \textcolor{blue}{1} + \textcolor{blue}{1} \\
10 &= \textcolor{blue}{6} + \textcolor{blue}{1} + \textcolor{blue}{1} + \textcolor{blue}{1} + \textcolor{blue}{1} \\
10 &= \textcolor{red}{7} + \textcolor{red}{2} + \textcolor{blue}{1} \\
10 &= \textcolor{red}{7} + \textcolor{red}{3} \\
10 &= \textcolor{blue}{10} \\
\end{aligned} $}
\\[-\rowheight]
\printcellmiddle
& \printcellmiddle
& \printcellmiddle
\\ \bottomrule
\end{tabular}
\label{Tab:Tcr}
\end{table}
%%%%%%%%%%

\noindent \textbf{Question:} \textit{Is the number of even admissible partitions of $n$ asymptotic to the number of odd admissible partitions of $n$ as $n$ tends to infinity?}\\

In this paper, our primary objective is to answer this question. In order to do so, we will bring into play the M\"{o}bius function. Actually, we will establish results for a more general version of this problem, where the M\"{o}bius function is replaced by M\"{o}bius convolutions. In the process, we are naturally led to prove results on exponential sums involving M\"{o}bius convolutions and related arithmetic functions. This is the other main goal of this article and it is of independent interest.

\subsection{Survey of previous results} 
In \cite{vaughanprimes}, Vaughan established an asymptotic formula for the number of partitions
of a number $n$ into primes when $n$ is large, by applying the Hardy-Littlewood circle method. In \cite{vaughansquares}, Vaughan then extended this approach to study partition asymptotics regarding squares and other additive questions.

Since the publication of \cite{vaughanprimes, vaughansquares}, there has been considerable interest in the study of restricted or weighted partitions over arithmetically interesting sets. In \cite{gafnipowers}, Gafni showed an asymptotic formula for the number of partitions of $n$ into $k$-th powers, and in \cite{gafniprimepowers}, she proved results for partitions into prime powers. Berndt, Malik and the third author \cite{bmz} studied partitions into $k$-th powers in arithmetic progressions. Dunn and the second author \cite{dunnrobles} then considered the more general case where the parts are values of an arbitrary polynomial $f$ of degree $d \ge 2$. More recently, Das, Berndt, Zeindler, and the second and third authors \cite{brzz, drzz, RoblesZeindler} studied partitions weighted by generalized divisor functions, partitions into semiprimes and almost primes. General strategies for dealing with asymptotic weighted partitions have been established by Meinardus  \cite{meinardus} and more recently by Debryune and Tenenbaum \cite{debruynetenenbaum}. Notice that in the case of the above results, the corresponding partitions have non-negative weights, unlike the question discussed in Section \ref{Motivation}. For such questions, where the weights take positive and negative values (see Figure \ref{fig:PsiMuMu}), one should be prepared in principle to have a main result which is an upper bound rather than an asymptotic formula. This indeed happens in our case, see Theorem \ref{Partition} below.

To proceed, we start by considering the partition function $p_{w}(n)$ associated to a generating function of the form,
\begin{align} \label{eq:introAux01}
\Psi_{w}(z) = \sum_{n=1}^\infty p_{w}(n) z^n = \prod_{n \in \N} (1-z^n)^{-w(n)}, \quad \lvert z \rvert <1,
\end{align}
where $w(n) \in \mathbb{Z}$ is the weight placed on each element $n \in \N$. We refer to such partitions as \textit{weighted} partitions. Consider the M\"{o}bius function denoted by $\mu$, which is a multiplicative function defined in the following way: for each prime $p$,
\[\mu(p^k) := \begin{cases}
  -1,  & k=1 \\
  0, & k \geqslant 2.
\end{cases}\]
By definition, $\mu^{2}(n)=1$ if $n$ is squarefree and $\mu^{2}(n)=0$ otherwise. Our cases of interest will have weights given by the M\"{o}bius function, i.e., $w(n)=\mu(n)$, or more generally by convolutions of the M\"{o}bius function, i.e., $w(n) = (\mu * \mu * \cdots * \mu)(n)$, where the Dirichlet convolution is performed $k$ times. 

It is easy to see that in the particular case when $w(n)= \mu(n)$, the partition function $p_{\mu}(n)$ is exactly the difference between the number of even admissible partitions and odd admissible partitions of $n$ as defined in Section \ref{Motivation}. More precisely, suppose we denote by $\mathcal{A}(n)$ the number of admissible partitions of $n$. Additionally, let $\mathcal{E}(n)$ and $\mathcal{O}(n)$ be the number of even and odd admissible partitions of $n$ respectively. Then
\begin{align}\label{difference}
p_{\mu}(n) &= \mathcal{E}(n)-\mathcal{O}(n),\\
\textrm{and} \quad \mathcal{A}(n) &= \mathcal{E}(n)+\mathcal{O}(n).
\end{align}
Therefore, in order to show that the number of even admissible partitions is asymptotic to the number of odd admissible partitions as $n$ tends to infinity, our goal in what follows is to provide an upper bound for their difference, i.e., for $p_{\mu}(n)$. 

Very recently, in \cite{daniels1, daniels2}, Daniels investigated partitions whose generating functions are given by
\begin{align} \label{eq:introAux02}
    \Psi(z,f) = 1 + \sum_{n=1} p(n,f) z^n = \prod_{n=1}^\infty (1- f(n)z^n)^{-1}, \quad \lvert z \rvert <1. 
\end{align}
Daniels refers to such partitions as \textit{signed} partitions. In particular, he studied the case when $f(n) = \mu(n)$. Note the difference between placing $\mu(n)$ in the exponent of the factor $(1-z^n)^{-\mu(n)}$ in \eqref{eq:introAux01} as opposed to the presence of $\mu(n)$ inside the factor $(1-\mu(n)z^n)^{-1}$ in \eqref{eq:introAux02}. This is a subtle difference, which surprisingly leads to the existence of main terms in Daniel's works and the non-existence of main terms in our context. For comparison, the reader is referred to Theorem 1.1 in \cite{daniels1} and Theorem \ref{Partition} below. 
We remark that Theorem \ref{Partition} below does lead to an asymptotic formula connecting the number of odd and even admissible partitions from Section \ref{Motivation}.  
\subsection{Main Results} We start by answering the question raised in Section \ref{Motivation}.

\begin{theorem}\label{asymptotic}
The number of odd admissible partitions of $n$ is asymptotic to the number of even admissible partitions of $n$, as $n$ tends to infinity.
\end{theorem}

More generally, for any positive integer $k$, we define the Popovici-M\"{o}bius function (see \cite[$\mathsection$2.2]{handbook2} as well as \cite{{CZ1991}} and \cite{KuhnRobles, RR2021} for related applications) with parameter $k$ by
\begin{align}\label{convolution defn}
    \mu_k(n) := (\mu * \mu * \cdots * \mu)(n),
\end{align}
where the Dirichlet convolution is performed $k$ times. The Dirichlet series for $\mu_k(n)$ is given by
\begin{equation*}
    \sum_{n=1}^\infty \frac{\mu_k(n)}{n^s} = \frac{1}{\zeta(s)^k},
\end{equation*}
for $\operatorname{Re}(s) > 1$. Let $p_{\mu_k}(n)$ be the number of partitions of $n$ weighted by $\mu_k(n)$, that is, 
\begin{align*}
    \sum_{n=1}^\infty p_{\mu_k}(n) z^n = \prod_{n=1}^\infty (1-z^n)^{-\mu_k(n)}.
\end{align*}
In the present paper, we will establish results on exponential sums twisted by M\"{o}bius convolutions for any positive integer $k$. Concerning applications to partitions, for simplicity, we will focus only on the case when $k=2$. This is more involved than the case when $k=1$, and already demonstrates all the ingredients necessary for a general $k$. For $k=2$, we prove the following result. 
\begin{theorem}\label{Partition}
For any fixed $B>0$ and for all $n \geqslant 2$, we have
\begin{align}\label{eq 1.6}
\log p_{\mu*\mu}(n) = O_B\left ( \frac{\sqrt{n}}{(\log n)^B} \right).
\end{align}
\end{theorem}
\begin{remark}
As we shall see in the proof of Theorem \ref{Partition}, the method can be generalized to any $k \geqslant 1$, with the same upper bound as in \eqref{eq 1.6}. In particular, when $k=1$, we have
\begin{align}\label{eq 1.7}
\log p_{\mu}(n) = O_B\left ( \frac{\sqrt{n}}{(\log n)^B} \right),
\end{align}
for any fixed $B>0$ and all $n \geqslant 2$. Theorem \ref{asymptotic} is an immediate consequence of \eqref{difference} and \eqref{eq 1.7}. See Figure \ref{fig:PsiMuMu1} for a plot of the logarithms of $\mathcal{A}(n)$ and $ p_{\mu}(n)$.
\end{remark}

\begin{remark}
    Results similar to Theorems \ref{asymptotic} and \ref{Partition} hold true if the M\"{o}bius function $\mu(n)$ is replaced by the Liouville function $\lambda(n)$, where
\[
\lambda(n) = \prod_{i=1}^{k} (-1)^{a_i}, \quad \textrm{where} \quad n = \prod_{i=1}^{k} p_i^{a_i}.
\]
\end{remark}

\begin{figure}[h]
    \centering
\includegraphics[scale=0.43]{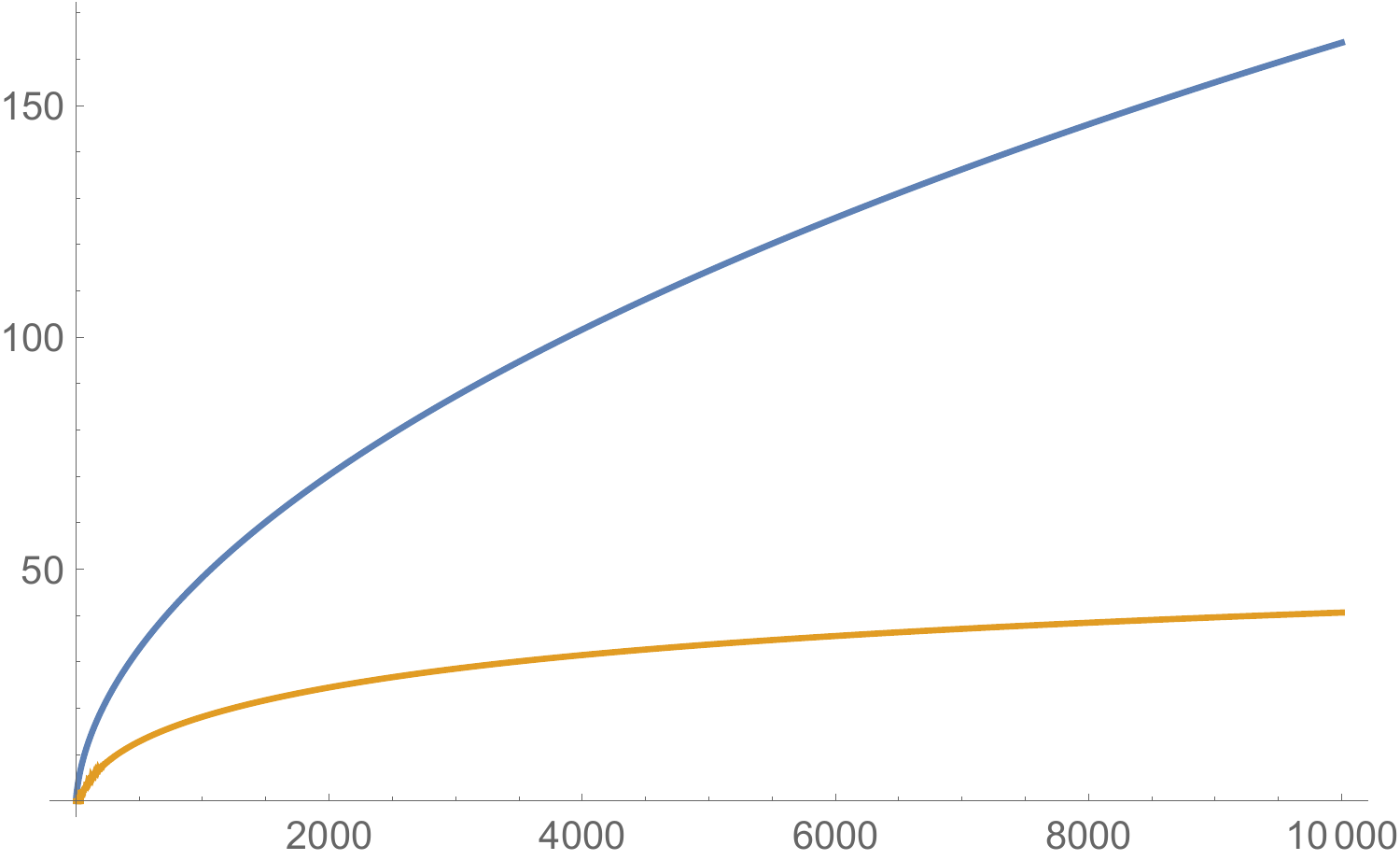}
\caption{Plot of $\log \mathcal{A}(n)$ in blue and $\log {p}_{\mu}(n)$ in orange for $1 \leqslant n \leqslant 10000$.}
    \label{fig:PsiMuMu1}
\end{figure}

\begin{figure}[h]
    \centering
\includegraphics[scale=0.32]{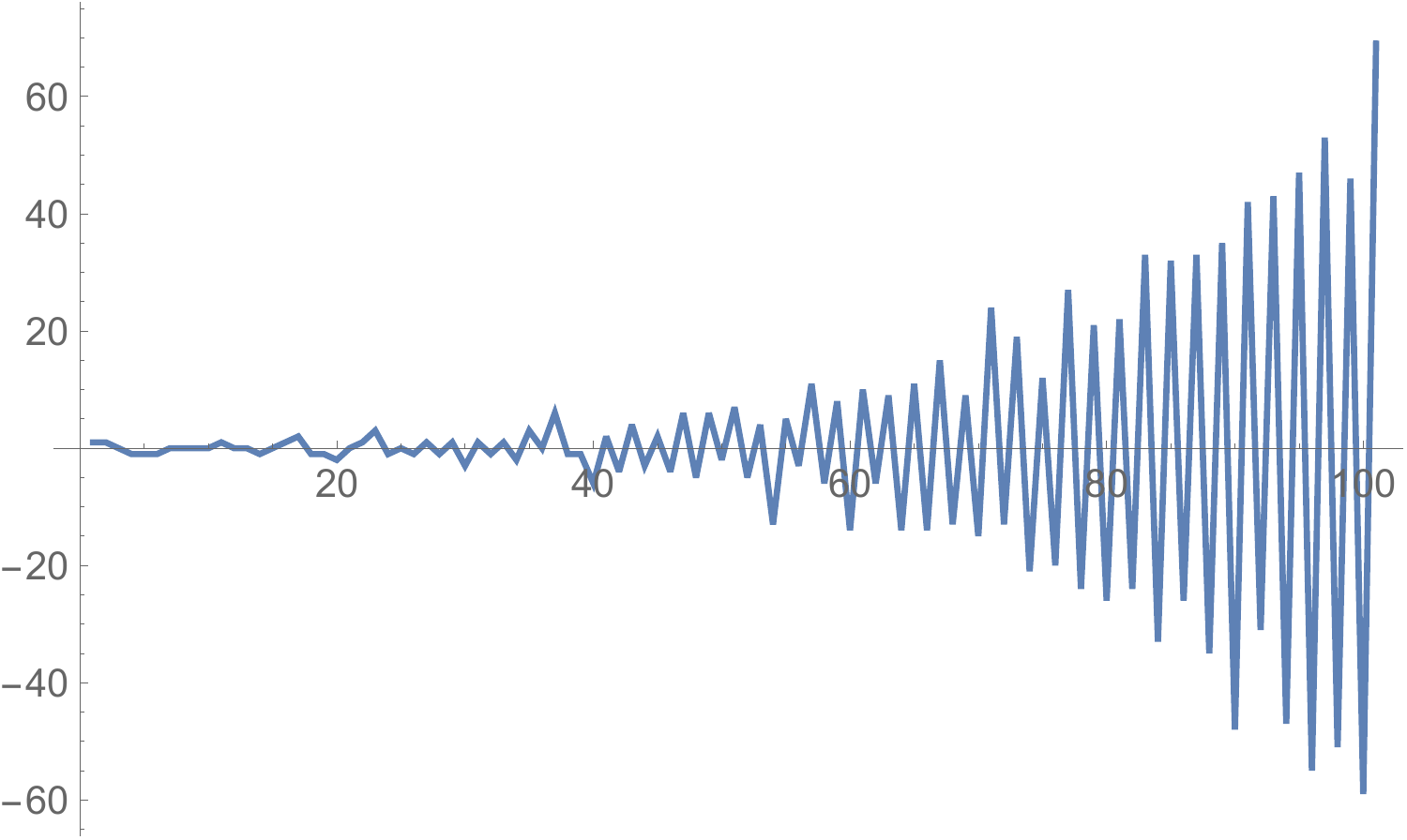}
    \includegraphics[scale=0.32]{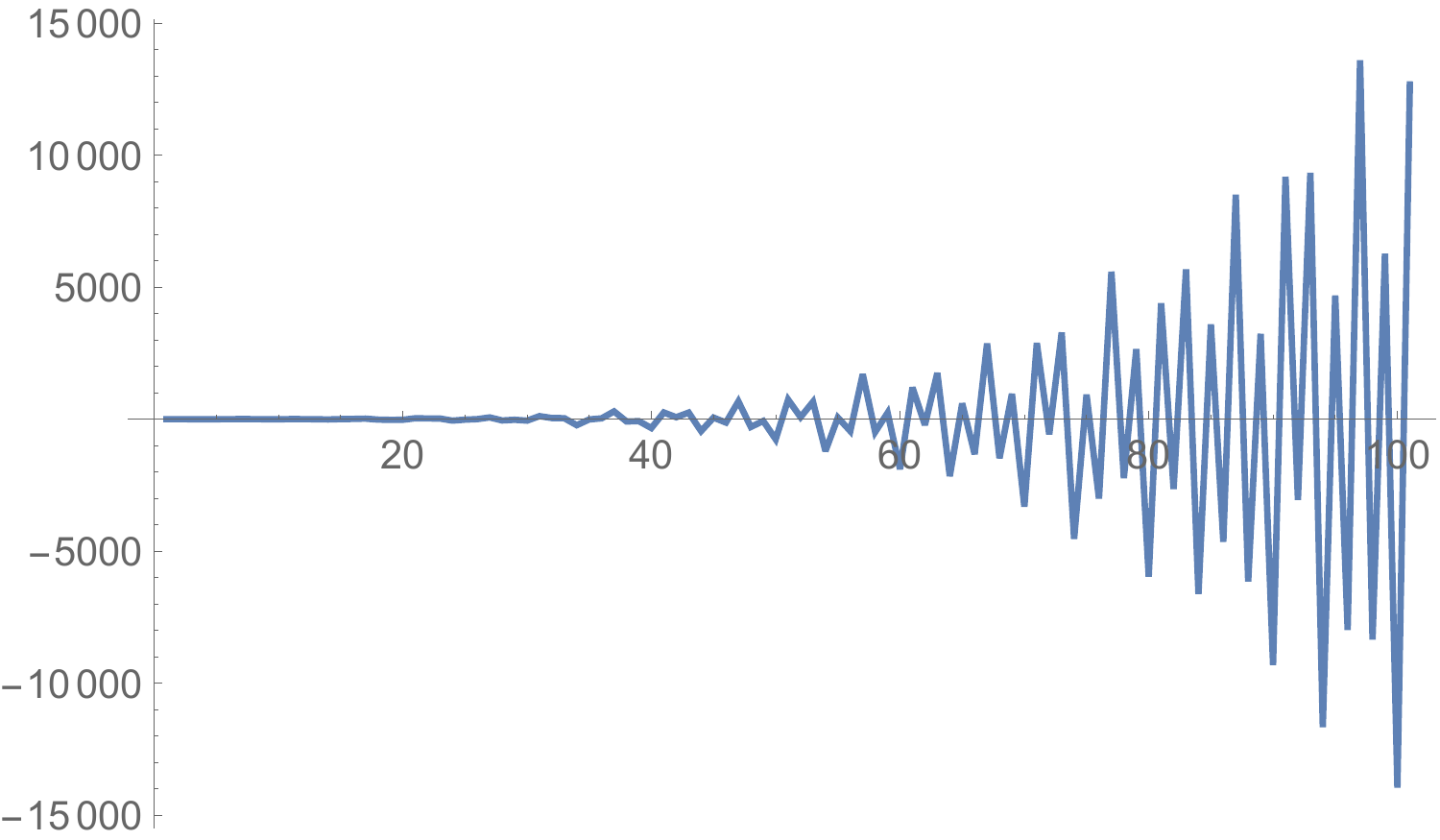}
\caption{\underline{Left}: Plot of partition function $p_w(n)$ with $w(n) = \mu(n)$, \underline{Right}: Plot of partition function $p_w(n)$ with $w(n) = (\mu*\mu)(n)$.}
    \label{fig:PsiMuMu}
\end{figure}

In order to prove Theorem \ref{Partition}, we will employ the Hardy-Littlewood method. As we shall see in Section \ref{sec:principalarc}, the contribution from the major arcs is closely related to the zeros of the Riemann zeta-function $\zeta(s)$. To be precise, following \eqref{eq:introAux01}, let us write 
\begin{align}\label{2 Convolution Mobius}
    \Psi_{\mu*\mu}(z) = \sum_{n=1}^\infty p_{\mu*\mu}(n) z^n = \prod_{n=1}^\infty (1-z^n)^{-(\mu*\mu)(n)}, \quad \lvert z \rvert <1.
\end{align}
Set
\begin{align}\label{eq 1.10}
    \Phi_{\mu*\mu}(z) := \sum_{j=1}^\infty \sum_{n=1}^\infty \frac{(\mu*\mu)(n)}{j} z^{jn}, \quad \lvert z \rvert <1.
\end{align}
Then $\Psi_{\mu*\mu}(z) = \exp(\Phi_{\mu*\mu}(z))$. Let $X \geqslant 1$ and $\rho = \exp(-1/X)$. Choose $z = \rho \e(\theta)$ in \eqref{eq 1.10} and let 
\[ \Delta := \frac{1}{(1+4\pi^2X^2\theta^2)^{1/2}}.
\]
Here, and in what follows, we write $\e(x)=\exp(2\pi i x)$. Then we have the following theorem.

\begin{theorem} \label{lem:principalmajor}
Suppose that $\theta \in \R$ and $X \ge 1$. If $X \Delta^3 \ge 1$, then there exists as sequence $T_{\nu}$ with $\nu \leqslant T_\nu \leqslant \nu+1$ such that
\begin{align} \label{eq:eqmainlemma}
    \Phi_{\mu * \mu}(\rho \e(\theta)) &= 4 \log \frac{X}{1-2\pi i X\theta} - 8 \log(2 \pi) + 72 \frac{1-2\pi i X\theta}{X}\nonumber \\
    & \quad + \lim_{\nu \to \infty} \sum_{|\operatorname{Im}(\rho)| < T_\nu}  f(X,\theta,\rho)  +\sum_{n=1}^{\infty} g(X,\theta,n),
\end{align}
where the functions $f$ and $g$ in \eqref{eq:eqmainlemma} are given by
\begin{align}
    f(X, \theta, \rho) &:= \bigg(\frac{X}{1-2\pi i X\theta}\bigg)^{\rho}\frac{\zeta (1 + \rho )\Gamma (\rho )}{{\zeta '{{(\rho )}^2}}} \bigg(\log \frac{X}{1-2\pi i X\theta} + \psi (\rho ) - \frac{{\zeta ''}}{{\zeta '}}(\rho ) + \frac{{\zeta '}}{\zeta }(1 + \rho ) \bigg), \nonumber \\
    g(X,\theta,n) &:= \bigg(\frac{X}{1-2\pi i X\theta}\bigg)^{-2 n} \bigg[\mathfrak{c}_1(n) \bigg(\log \frac{X}{1-2\pi i X\theta}\bigg)^2 + \mathfrak{c}_2(n) \log \frac{X}{1-2\pi i X\theta}+ \mathfrak{c}_3(n) \bigg]. \nonumber 
\end{align}
The coefficients $\mathfrak{c}_1(n), \mathfrak{c}_2(n)$ and $\mathfrak{c}_3(n)$ are defined in \eqref{eq:defc1}, \eqref{eq:defc2} and \eqref{eq:defc3} respectively. The sum is taken over the non-trivial zeros $\rho$ of $\zeta(s)$ under the assumption, for notational ease, of simplicity.
\end{theorem}

\begin{remark}
As Titchmarsh puts it in \cite[$\mathsection$14.27]{Titchmarsh}: `obvious modifications are required if' the zeros are not simple. Indeed, this assumption does not represent a technical difficulty and it can be relaxed at the expense of cluttering the expression $f(X,\theta,\rho)$ for the non-trivial zeros.
\end{remark}

\begin{remark}
Interestingly, \eqref{eq:eqmainlemma} is an instance of an explicit formula in the theory of partitions. It is reminiscent of the classical explicit formula
\begin{align*}
    \psi_0(x) = \frac{1}{2} \lim_{\delta \to 0} (\psi(x+\delta)-\psi(x-\delta)) = x - \lim_{T \to \infty}\sum_{\lvert \operatorname{Im}(\rho) \rvert <T} \frac{x^\rho}{\rho} - \log(2\pi) + \sum_{n=1}^{\infty} \frac{x^{-2n}}{2n}.
\end{align*}
Here $\rho$ runs over the non-trivial zeros of $\zeta(s)$ and
\begin{align}\label{Psi defn}
    \psi(x) = \sum_{p^k \leqslant x} \log p,
\end{align}
where the sum in \eqref{Psi defn} is over prime powers.
\end{remark}
\begin{remark}
Consider the truncated arithmetic sum
\begin{align*} %\label{eq:PhiArithemticTruncated}
    \Phi_1 := \Phi_{1}(X,\theta, J, N) = \sum_{j=1}^J \sum_{n=1}^N \frac{(\mu * \mu)(n)}{j} \exp\bigg( -jn \bigg(\frac{1}{X}-2\pi i \theta\bigg)\bigg),
\end{align*}
and its analytic truncated counterpart
\begin{align*} 
\Phi_2 := \Phi_{2}(X,\theta,T,N) &= 4 \log \frac{X}{1-2\pi i X\theta} - 8 \log(2 \pi) + 72 \frac{1-2\pi i X\theta}{X} \\
&\quad+ \sum_{\lvert \operatorname{Im}(\rho) \rvert < T}  f(X,\theta,\rho)  +\sum_{n=1}^{N} g(X,\theta,n) .
\end{align*}
We set $\theta=\theta(X)$ as
\begin{align*}
    \theta(X) := \frac{1}{2\pi} \sqrt{\frac{1}{X^{4/3}}-\frac{1}{X^2}},
\end{align*}
such that $X \Delta^3=1$. In Figures \ref{fig:Renumericallemmaprincipal} and \ref{fig:Imnumericallemmaprincipal}, we numerically illustrate our result from Theorem \ref{lem:principalmajor} in terms of the functions $\Phi_1$ and $\Phi_2$.
\end{remark}

\begin{figure}[h!]
    \centering
\includegraphics[scale=0.645]{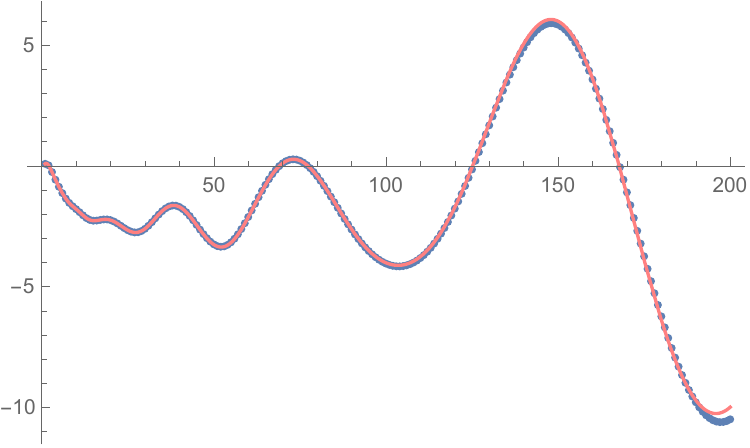}
\includegraphics[scale=0.645]{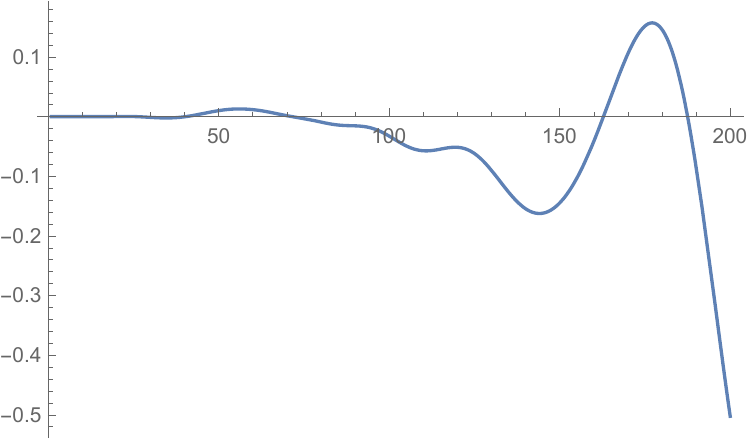}
\caption{\underline{Left}: Plot of $\operatorname{Re}(\Phi_1(X,\theta,120,800))$ in blue and $\operatorname{Re}(\Phi_2(X,\theta,20,10))$ in pink. \underline{Right}: Plot of $\operatorname{Re}(\Phi_1(X,\theta,120,800))-\operatorname{Re}(\Phi_2(X,\theta,20,10))$.}

\label{fig:Renumericallemmaprincipal}
%\end{figure}
%
%\begin{figure}[h!]
\centering
\includegraphics[scale=0.645]{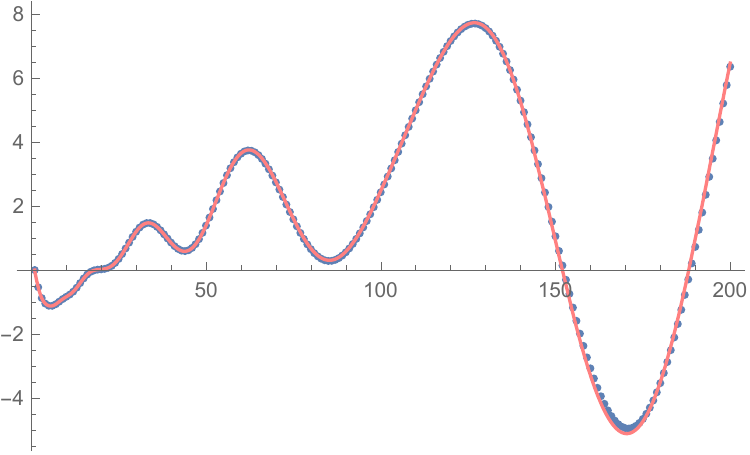}
\includegraphics[scale=0.645]{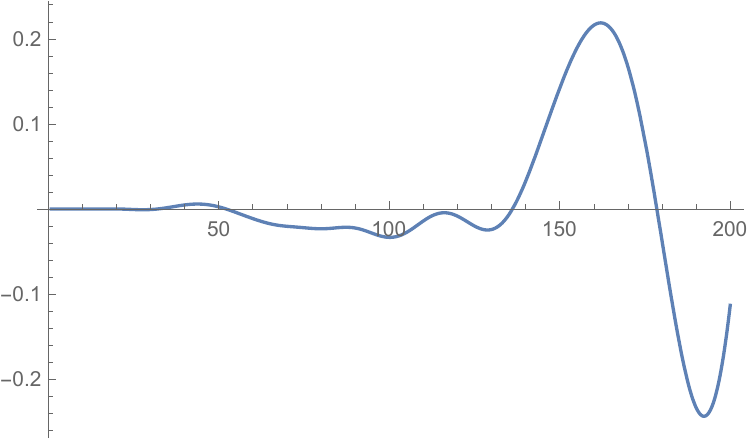}
\caption{\underline{Left}: Plot of $\operatorname{Im}(\Phi_1(X,\theta,120,800))$ in blue and $\operatorname{Im}(\Phi_2(X,\theta,20,10))$ in pink. \underline{Right}: Plot of $\operatorname{Im}(\Phi_1(X,\theta,120,800))-\operatorname{Im}(\Phi_2(X,\theta,20,10))$.}
\label{fig:Imnumericallemmaprincipal}
%\end{figure}
%
%\begin{figure}[h!]
\label{fig:numericallemmaprincipal}
\end{figure}

We will also need to estimate the minor arc contributions arising from the Hardy-Littlewood method. This leads us to the second major goal of the present paper, which is to study exponential sums twisted by general M\"{o}bius convolutions and related questions. With this in mind, we start by considering the M\"{o}bius function itself. Let us write
\begin{align*}
    S_{\mu}(X, \alpha) =\sum_{n\leqslant X}\mu(n)\e(n\alpha).
\end{align*}
Then we have the following result.
\begin{theorem}\label{theorem 1}
Let $\alpha\in \mathbb{R}$ and consider a reduced fraction $a/q$ such that $|\alpha-a/q|\leqslant1/q^2$. Then for $X \geqslant 2$ and for any fixed $\varepsilon>0$, we have
\begin{align*}
S_{\mu}(X, \alpha) \ll_{\varepsilon} X^{4/5+\varepsilon}+ \dfrac{X(\log X)^3}{q^{1/2}} + X^{1/2}q^{1/2}(\log X)^3.
\end{align*}
\end{theorem}
More generally, let 
\begin{align*}
    S_{\mu_k}(X, \alpha) =\sum_{n\leqslant X }\mu_k(n)\e(n\alpha),
\end{align*}
where $\mu_k(n)$ is as defined in \ref{convolution defn}. We prove the following theorem.
\begin{theorem}\label{Mobius-k convolutions}
Let $\alpha\in \mathbb{R}$ and consider a reduced fraction $a/q$ such that $|\alpha-a/q|\leqslant1/q^2$. Consider the sequences $\{a_m\}, \{b_m\}$ and $ \{c_m\}$ given by
\begin{align*}
    a_1 &= \frac{4}{5} \quad \textrm{and} \quad \hspace{0.35cm}a_{m+1} = \frac{4-a_m}{5-2a_m} \quad \textrm{for} \quad m=1,2,\ldots, \\
    b_1 &= \frac{1}{2} \quad \textrm{and} \quad \quad b_{m+1} = \frac{b_m}{3} \quad \quad \quad \hspace{0.05cm} \textrm{for} \quad m=1,2,\ldots, \\
    \textrm{as well as} \quad c_1 &= \frac{1}{2} \quad \textrm{and} \quad \quad c_{m+1} = \frac{4-c_m}{5-2c_m} \quad \textrm{for} \quad m=1,2,\ldots.
\end{align*}
Then for $X \geqslant 2$, $k \geqslant 2$ and any fixed $\varepsilon>0$, we have
\begin{align*}
    S_{\mu_k}(X, \alpha) \ll_{\varepsilon} X^{a_k+\varepsilon}+ \dfrac{X(\log X)^{k^2}}{q^{b_k}} + X^{c_k}q^{1-c_k}(\log X)^{k^2}.
\end{align*}
\end{theorem}

We also establish results for exponential sums involving Dirichlet convolutions of the M\"{o}bius function with the characteristic function of squarefull and squarefree integers respectively. To this end, we define the arithmetic function $\hat{\mu}$ by
\begin{align*}
    \hat{\mu}(n) := (\mu*f)(n),
\end{align*}
where $f$ denotes the indicator function for the squarefull numbers. Writing 
\begin{align*}
    S_{\hat{\mu}}(X, \alpha) =\sum_{n\leqslant X }\hat{\mu}(n)\e(n\alpha), 
\end{align*}
we are able to prove the following estimate.
\begin{theorem}\label{theorem 2}
Let $\alpha\in \mathbb{R}$ and consider a reduced fraction $a/q$ such that $|\alpha-a/q|\leqslant1/q^2$. Then for $X \geqslant 2$, we have
\begin{align*}
    S_{\hat{\mu}}(X, \alpha) \ll \bigg(X^{5/6} + q^{1/2}X^{1/2} + \dfrac{X}{q^{1/2}}\bigg)(\log X)^{5/2}.
\end{align*}
\end{theorem}

Considering the convolution of the M\"{o}bius function with the characteristic function of squarefree integers, we define $\tilde{\mu}_k(n)$ in terms of coefficients of the following Dirichlet series. Let $k \geqslant 1$ and
\begin{align*}
    \sum_{n=1}^\infty\dfrac{\tilde{\mu}_k(n)}{n^s}=\dfrac{1}{\zeta(ks)},
\end{align*}
for $\operatorname{Re}(s) > 1/k$. As an example, for $k=2$, one has $\tilde{\mu}_2(n) = (\mu*\mu^2)(n)$, where $\mu^2$ is the indicator function for the squarefrees. Our next result provides an upper bound for the exponential sum
\begin{align*}
 S_{\tilde{\mu}_2}(X, \alpha) = \sum_{n\leqslant X }\tilde{\mu}_2(n)\e(n\alpha).
\end{align*}
\begin{theorem}\label{theorem 3}
Let $\alpha\in \mathbb{R}$ and consider a reduced fraction $a/q$ such that $|\alpha-a/q|\leqslant1/q^2$. Then for $X \geqslant 2$, we have
\begin{align*}
S_{\tilde{\mu}_2}(X, \alpha)\ll \bigg(X^{23/28} + \dfrac{X}{q^{1/4}} + X^{3/4}q^{1/4}\bigg)(\log X)^{10}.
\end{align*}
\end{theorem}
\begin{remark}
One can apply the method from the proof of Theorem \ref{Mobius-k convolutions} to generalize the result in Theorem \ref{theorem 3} to any $k \geqslant 1$. 
\end{remark}
\subsection{Structure of the paper} The paper is organized as follows. In Section \ref{sec:preliminaries}, we introduce some preliminary lemmas, such as Type I and Type II estimates on exponential sums involving Dirichlet convolutions, a result on the distribution of the M\"{o}bius function in arithmetic progressions, a divisor sum estimate and an unconditional upper bound on $1/\zeta(s)$ in the critical strip along specific vertical line segments. Section \ref{sec:proofsexpsums} is devoted to the proofs of Theorems \ref{theorem 1}, \ref{theorem 2} and \ref{theorem 3}. The proof of Theorem \ref{Mobius-k convolutions} is singled out in Section \ref{sec:secondexpsumproof}. In Section \ref{sec:partitions}, we move on to the theory of partitions and set up the proof of Theorem \ref{Partition}. Section \ref{sec:nonprincipal} and Section \ref{sec:minorarcs} will contain the details of the major and minor arc estimates respectively. Finally, in \ref{sec:principalarc}, we present the proof of Theorem \ref{lem:principalmajor}.

\section{Preliminary Lemmas} \label{sec:preliminaries}
In this section, we start by establishing some standard notations. Then we collect some necessary results from the literature and proceed to prove some preliminary lemmas. 
\subsection{Notations} We employ some standard notation that will be used throughout the paper.

\begin{itemize}
    \item Throughout the paper, the expressions $f(X)=O(g(X))$, $f(X) \ll g(X)$, and $g(X) \gg f(X)$ are equivalent to the statement that $|f(X)| \leq C|g(X)|$ for all sufficiently large $X$, where $C>0$ is an absolute constant. A subscript of the form $\ll_{\alpha}$ means the implied constant may depend on the parameter $\alpha$. Dependence on several parameters is indicated in an analogous manner, as in $\ll_{\alpha, \lambda}$.
    \item  The notation $f = o(g)$ as $x\to a$  means
that $\lim_{x\to a} f(x)/g(x) = 0$ and $f\sim g$ as $x\to a$ denotes $\lim_{x\to a} f(x)/g(x)= 1$.
    \item We denote the set of real numbers by $\mathbb{R}$, the set of rational numbers by $\mathbb{Q}$, the set of integers by $\mathbb{Z}$, the set of natural numbers by $\mathbb{N}$ and the set of complex numbers by $\mathbb{C}$.
\item For any set $\mathcal{A}$, $\operatorname{card}(\mathcal{A})$ denotes the cardinality of the set $\mathcal{A}$.
\item The notation $\e(x)$ stands for $\exp(2\pi i x)$.
\item Dyadic sums are represented by $\sum_{n \sim N} f(n) = \sum_{N < n \leqslant 2N} f(n)$.
\item The divisor function is denoted by $d(n)$. The three-fold divisor function $\tau_3(n)$ is defined by the coefficients of the Dirichlet series 
\[\zeta^3(s) = \sum_{n=1}^{\infty}\frac{\tau_3(n)}{n^s}, \quad \textrm{for  } \operatorname{Re}(s) > 1.
\]
\end{itemize}
\subsection{Exponential sums and Dirichlet convolutions}
We first state the following two lemmas due to Koukoulopoulos, that provide bounds on exponential sums involving Dirichlet convolutions.
\begin{lemma}[Type I estimate]\label{lem1}
Let $f\colon \mathbb{N}\to \mathbb{C}$ be supported on $[1, y]$, $v\geqslant 2$, $X\geqslant 2$, $\alpha\in \mathbb{R}$ and $a/q$ be a reduced fraction such that $|\alpha-a/q|\leqslant 1/q^2$. Then
\begin{align*}
\sum_{n\leqslant X}(f*\log^v)(n)\e(n\alpha)\ll \bigg(y + \dfrac{X}{q}+q\bigg)(\log X)^{v+1}\|f\|_{\infty}.
\end{align*}
\end{lemma}
\begin{proof}
See \cite[Theorem 23.5]{K2019}.
\end{proof}

\begin{lemma}[Type II estimate]\label{lem2}
Let $f, g\colon \mathbb{N}\to \mathbb{C}$ be two arithmetic functions supported on $[1, y]$ and $[1, z]$, respectively. Let $X\geqslant 2$, $\alpha\in \mathbb{R}$ and consider a reduced fraction $a/q$ such that $|\alpha-a/q|\leqslant 1/q^2$. Then
\begin{align*}
\sum_{n\leqslant X}(f*g)(n)\e(n\alpha)\ll \bigg(q + y +z + \dfrac{yz}{q}\bigg)^{1/2}\sqrt{\log 2q}\cdot \|f\|_2\|g\|_2.
\end{align*}
\end{lemma}
\begin{proof}
    See \cite[Theorem 23.6]{K2019}.
\end{proof}

We would like to establish more general versions of Lemmas \ref{lem1} and \ref{lem2}, in the case when $\alpha\in \mathbb{R}$ and there exists some reduced fraction $a/q$ such that $|\alpha-a/q|\leqslant\gamma/q^2$, for some $\gamma \geqslant 1$. In this regard, we first prove the following proposition.

\begin{proposition}\label{Prop7}
Let $\alpha\in \mathbb{R}$ and consider a reduced fraction $a/q$ such that $|\alpha-a/q|\leqslant\gamma/q^2$ for some $\gamma \geqslant 1$. Then for each $m_1 \in \{1,2,\dots,q\}$, there exists at most $14 \gamma$ choices of $m_2 \in \{1,2,\dots,q \}$ such that
\begin{align} \label{eq:auxiliaryexpsum1}
\bigg \lvert \norm{m_1\alpha}- \norm{m_2\alpha} \bigg \rvert < \frac{1}{q},
\end{align}
where $\norm{x}$ denotes the smallest distance of $x$ to an integer.
\end{proposition}
\begin{proof}
Fix $m_1 \in \{1,2,\dots, q\}$ and suppose $m_2 \in \{1,2\dots,q\}$ such that \eqref{eq:auxiliaryexpsum1} is satisfied.
%\begin{align}\label{Eq 1}
%\bigg \lvert \norm{m_1\alpha}- \norm{m_2\alpha} \bigg \rvert < \frac{1}{q}.
%\end{align}
Writing $\alpha=a/q+\beta$ with $\lvert \beta \rvert  \leqslant\gamma/q^2$, we have
\[m_1 \alpha =  \frac{m_1a}{q}+m_1\beta \quad \textrm{ and} \quad m_2 \alpha =  \frac{m_2a}{q}+m_2\beta. \]
Therefore, we obtain
\begin{align}\label{Eq 2}
\bigg \lvert \norm{m_1\alpha}- \norm{\frac{m_1 a}{q}} \bigg \rvert \leqslant\lvert m_1\beta \rvert \leqslant\frac{\gamma}{q} 
\quad \textnormal{and} \quad 
\bigg \lvert \norm{m_2\alpha}- \norm{\frac{m_2 a}{q}} \bigg \rvert \leqslant\lvert m_2\beta \rvert \leqslant\frac{\gamma}{q}.
\end{align}
%and
%\begin{align}\label{Eq 3}
%\bigg \lvert \norm{m_2\alpha}- \norm{\frac{m_2 a}{q}} \bigg \rvert \leqslant\lvert m_2\beta \rvert \leqslant\frac{\gamma}{q}.
%\end{align}
From \eqref{eq:auxiliaryexpsum1} and \eqref{Eq 2}, 
%and \eqref{Eq 3}
it follows that
\begin{align}\label{Eq 4}
\bigg \lvert \norm{\frac{m_1 a}{q}}- \norm{\frac{m_2 a}{q}} \bigg \rvert \leqslant \frac{2\gamma+1}{q} \leqslant\frac{3\gamma}{q}.
\end{align}
Next, we write 
\[
\norm{\frac{m_1 a}{q}} = \frac{\ell_1}{q} \quad \textnormal{and} \quad \norm{\frac{m_2 a}{q}} = \frac{\ell_2}{q},
\]
where $\ell_1, \ell_2 \in \{0,1,\dots,\lfloor q/2\rfloor \}$. Note that when $m_1$ is fixed, $\ell_1$ is fixed. Also, for each choice of $\ell_2$, there's at most two different choices of $m_2$. Hence, in order to satisfy \eqref{Eq 4}, there's at most $7\gamma$ choices of $\ell_2$. This implies there are at most $14 \gamma$ choices of $m_2$ which satisfy \eqref{eq:auxiliaryexpsum1}.
\end{proof}

With Proposition \ref{Prop7} in hand, we have the following two lemmas.

\begin{lemma}\label{lem5}
Let $f\colon \mathbb{N}\to \mathbb{C}$ be supported on $[1, y]$, $v\geqslant 2$, $X\geqslant 2$, $\alpha\in \mathbb{R}$ and $a/q$ be a reduced fraction such that $|\alpha-a/q|\leqslant\gamma/q^2$ for some $\gamma \geqslant 1$. Then
\begin{align}\label{Type 1 with Gamma}
    \sum_{n\leqslant X}(f*\log^v)(n)\e(n\alpha)\ll \gamma \bigg(y + \dfrac{X}{q}+q\bigg)(\log x)^{v+1}\|f\|_{\infty},
\end{align}
where the implied constant in \eqref{Type 1 with Gamma} is absolute.
\end{lemma}

\begin{lemma}\label{lem6}
Let $f, g\colon \mathbb{N}\to \mathbb{C}$ be two arithmetic functions supported on $[1, y]$ and $[1, z]$, respectively. Let $X\geqslant 2$, $\alpha\in \mathbb{R}$ and consider a reduced fraction $a/q$ such that $|\alpha-a/q|\leqslant\gamma/q^2$ for some $\gamma \geqslant 1$. Then
\begin{align}\label{Type 2 with Gamma}
    \sum_{n\leqslant X}(f*g)(n)\e(n\alpha)\ll \gamma \bigg(q + y +z + \dfrac{yz}{q}\bigg)^{1/2}\sqrt{\log 2q}\cdot \|f\|_2\|g\|_2,
\end{align}
where the implied constant in \eqref{Type 2 with Gamma} is absolute.
\end{lemma}
We only prove Lemma \ref{lem5}. The proof of Lemma \ref{lem6} is similar.

\begin{proof}[Proof of Lemma \textnormal{\ref{lem5}}]
Following the proof of Theorem 23.5 in \cite{K2019}, it suffices to bound the sum
\[\sum_{k \leqslant y} \min \bigg \{\frac{X}{k},\frac{1}{\norm{k\alpha}} \bigg \}.
\]
To do so, we divide the interval $[1,y]$ into subintervals of length $q$ and write
$$
S_m:=\sum_{m q<k \leqslant(m+1) q} \min \bigg \{\frac{X}{k},\frac{1}{\norm{k\alpha}} \bigg \}.
$$
Note that 
\[\sum_{k \leqslant y} \min \bigg \{\frac{X}{k},\frac{1}{\norm{k\alpha}} \bigg \} \leqslant\sum_{m=0}^{\lfloor y/q \rfloor} S_m.
\]
We write $\alpha=a/q+\beta$ with $\lvert \beta \rvert  \leqslant \gamma/q^2$ and consider two cases.\\

\noindent \textbf{Case 1: $m \geqslant 1$.} Writing $k=mq+r$ where $1 \leqslant r \leqslant q$, we have $
\norm{k\alpha} = \norm{r\alpha+mq\beta}$. Define the intervals 
\[I _j = \bigg [\frac{j-1}{2q}, \frac{j}{2q}\bigg ), \quad j \in \{1,2,\dots, \lceil q \rceil \}.
\]
Fix some $k_1=mq+r_1$, such that $ \norm{k_1\alpha} \in I _j$ for some $j$. Since $I _j$ is an interval of length $1/2q$, it follows that $\norm{r_1\alpha}$ is contained in the union of at most two intervals, say $T_1$ and $T_2$, each of length at most $1/2q$. Suppose for some  $k_2=mq+r_2$, $ \norm{k_2\alpha} \in I _j$ as well. Then $\norm{r_2\alpha}$ is also contained in $T_1 \cup T_2$. But using Proposition \ref{Prop7}, each of $T_1$ or $T_2$ can contain at most $14 \gamma$ choices of $r_2$. Hence, there can be at most $28 \gamma$ different $k_2$ such that both $\norm{k_1 \alpha}$ and $\norm{k_2 \alpha}$ lies in $I _\ell$. Therefore, the contribution from those $k$'s for which $\norm{k \alpha} \in I _j$ to the sum $S_m$ is
\[\leqslant\frac{60 \gamma q}{j-1},
\]
when $j >1$. For the specific case when $j=1$, we trivially bound the contribution by $30\gamma/mq$. Therefore, for $m \geqslant 1$, we obtain

$$
S_m \leqslant 30 \gamma \bigg (\frac{ X}{m q}+\sum_{2 \leqslant j \leqslant \lceil q \rceil } \frac{2 q}{j-1} \bigg) \ll \gamma \left (\frac{ X}{m q}+q \log q \right ) .
$$

\noindent \textbf{Case 2: $m=0$.} For all $r \geqslant q/2\gamma$, an argument similar to the previous case shows that the contribution from such terms is $\ll \gamma (X/q+q \log q)$. Finally, we are left with the cases when $r<q/2\gamma$. Then $\lvert r \beta \rvert \leqslant1/2q$, which implies $\norm{r \alpha} \geqslant 1/2q$. Therefore, the total contribution from the cases when $r < q/2\gamma$ is 
\[
\ll \gamma(2q+q \log q) \ll \gamma q \log q.
\]
The rest of the proof follows along the lines of the proof of \cite[Theorem 23.5]{K2019}, and noting that there are $\ll y/q$ choices of $m$.  
\end{proof}
\subsection{Exponential sums and the divisor function}
We record the following result achieved by Mikawa that provide bounds on exponential sums involving the three-fold divisor function.
\begin{lemma}\label{lem: Mik}
Let $\alpha\in \mathbb{R}$ and consider a reduced fraction $a/q$ such that $|\alpha-a/q|\leqslant1/q^2$. Then, for $0<M,$ $J\leqslant X $, we have
\[\sum_{m\sim M}\sum_{j\sim J}\tau_{3}(j)\min\bigg(\frac{X}{m^{2}j},\frac{1}{\Vert\alpha m^{2}j\Vert}\bigg) \ll MJ(\log X)^{3}+\dfrac{X^{3/4}}{M}\bigg(Xq^{-1}+XM^{-1}+q\bigg)^{1/4}(\log X)^{8},\]
where $\tau_3$ denotes the three-fold divisor function.
\end{lemma}
\begin{proof}
See \cite{M2000}.
\end{proof}
\subsection{The M\"{o}bius function in arithmetic progressions}
We collect below a result on the M\"{o}bius function in arithmetic progressions, due to Davenport, which is an application of the Siegel-Walfisz theorem with the standard zero-free region for Dirichlet $L$-functions.
\begin{lemma}\label{Davenport}
Let $X\geqslant 2$ and $B>0$. Let $r, q\geqslant 1$ be two integers with $q\leqslant(\log X)^B$. Then there exists a constant $c_B>0$ such that
\begin{align*}
    \sum_{\substack{n \leqslant X \\ n \equiv r \modu q}} \mu(n) \ll_B X\exp(-c_B\sqrt{\log X}),
\end{align*}
where $c_B>0$ is a constant depending only on $B$.
\end{lemma}
\begin{proof}
See \cite[Lemma 6]{D1937}.
\end{proof}
\subsection{A divisor sum estimate} Let $k \in \mathbb{N}$ and define
\[d_k(n) := (1 * 1 *\cdots * 1)(n),\]
where the Dirichlet convolution of the identity function $1$ is performed $k$ times. The Dirichlet series for $d_k(n)$ is given by $\zeta(s)^k$. Norton established the following result.
\begin{lemma}\label{Divisor Sum Powers}
For $X \geqslant 1$, and $k, r \geqslant 2$, we have
\begin{align*}
    \sum_{n \leqslant X } (d_k(n))^r \ll X(\log X)^{k^r-1}.
\end{align*}
\begin{proof}
See \cite[Eq. 1.10]{N1990}.
\end{proof}
\end{lemma}
\subsection{Bounding $1/\zeta(s)$ in the critical strip.}
In order to prove the explicit formula in Section \ref{sec:principalarc}, we will require a contour integration involving the term $1/\zeta^2(s)$. To prove our results unconditionally, we shall use the following result due to works of Ramachandra and Sankaranarayanan, and Inoue, which provides unconditional upper bounds for $1/\zeta(s)$ inside the critical strip along a specific vertical line segment.
\begin{lemma} \label{lem:RamachandraSankaranarayanan}
Let $T>0$ be a sufficiently large positive number and $H=T^{1/3}$. Then one has
\begin{align*}
        \min_{T \leqslant t \leqslant T+H} \max_{\frac{1}{2} \leqslant \sigma \leqslant 2} |\zeta(\sigma+it)|^{-1} \leqslant \exp( C (\log \log T)^2)
\end{align*}
with an absolute constant $C>0$. In particular, there exists a real $T_* \in [T,T+T^{1/3}]$ such that
\begin{align*}
        \frac{1}{\zeta(\sigma+iT_*)} \leqslant T_*^\varepsilon \quad \textnormal{where} \quad \frac{1}{2} \leqslant \sigma \leqslant 2,
\end{align*}
for any $\varepsilon>0$.
\end{lemma}
\begin{proof}
See \cite[Theorem 2]{RamachandraSankaranarayanan} and \cite[Lemma 1]{Inoue}.
\end{proof}
\section{Proofs of Theorems \ref{theorem 1}, \ref{theorem 2} and \ref{theorem 3}} \label{sec:proofsexpsums}
In this section we present the proofs of Theorems \ref{theorem 1}, \ref{theorem 2} and \ref{theorem 3}.

\subsection{Proof of Theorem \ref{theorem 1}} We can assume $q \leqslant X$, otherwise, the conclusion follows trivially. Let $U, V \geqslant 2$ be parameters to be chosen later on. Applications of the Möbius inversion formula yield
\begin{align}\label{Mobius : Step 2}
 \mu(n) &= \mu_{\leqslant U}(n) + \mu_{\leqslant V}(n)+  (\mu_{>U}*\mu_{>V}*1)(n) -(\mu_{\leqslant U}*\mu_{\leqslant V}*1)(n) \notag \\
 & = \mu^{(1)} + \mu^{(2)}+\mu^{(3)}-\mu^{(4)},
\end{align}
say. Consider $\mu^{(1)}$ and $\mu^{(2)}$ first. Since our choice of parameters $U$ and $V$ will be small compared to $X$, bounding trivially, we arrive at
\begin{align}%\label{mu-1 sum}
\sum_{n\leqslant X }\mu^{(1)}(n)\e(n\alpha)&\ll U \label{mu-1 sum} \\
     \sum_{n\leqslant X }\mu^{(2)}(n)\e(n\alpha)& \ll V \label{mu-2 sum}.
\end{align}
For $\mu^{(3)}(n)$, we will use Lemma \ref{lem2}. However, applying the lemma directly doesn't suffice our purpose. We consider breaking our range into dyadic intervals. By definition, we have
$$
\mu^{(3)}(n)=\sum_{\substack{k \ell=n \\ k>U, \ell >V}} \mu_{>U}(k) (\mu_{>V} * 1)(\ell) .
$$
To obtain better control of the support of the variables $k$ and $\ell$, we break the interval $(U, X/ V]$ into dyadic intervals $\left(2^{j-1}, 2^j\right]$, where $2^j \in(U, 2 X / V]$. If $k \in\left(2^{j-1}, 2^j\right]$, then $\ell=n / k \leqslant X / 2^{j-1}$. With this more accurate decomposition, we have
$$
\mu^{(3)}(n)=\sum_{U<2^j \leqslant 2 X / V}\left(f_j * g_j\right)(n), \quad \text { for } n \leqslant X ,
$$
where $f_j(k)=\mu_{>U}(k) \cdot 1_{2^{j-1}<k \leqslant 2^j}$ and $g_j(\ell)=(\mu_{>V} * 1)(\ell) \cdot 1_{V<\ell \leqslant X/ 2^{j-1}}$. Then, applying Lemma \ref{lem2}, we obtain 
\begin{align}\label{mu-3 sum}
    \sum_{n\leqslant X }\mu^{(3)}(n)\e(n\alpha) &= \sum_{U<2^j \leqslant 2 X / V} \sum_{n\leqslant X }\left(f_j * g_j\right)(n)\e(n\alpha) \notag \\
    &\ll \sum_{U<2^j \leqslant 2 X / V} \bigg(q^{1/2} + \dfrac{X^{1/2}}{U^{1/2}} +\dfrac{X^{1/2}}{V^{1/2}} + \dfrac{X^{1/2}}{q^{1/2}}\bigg)\sqrt{\log 2q}\cdot \|f_j\|_2\|g_j\|_2 \notag \\
    &\ll \bigg(X^{1/2}q^{1/2} + \dfrac{X}{U^{1/2}} +\dfrac{X}{V^{1/2}}+ \dfrac{X}{q^{1/2}}\bigg)(\log X)^3.
\end{align}

\noindent \textbf{Note.} In the rest of this section, we will always apply the above more accurate decomposition technique before applying Lemma \ref{lem2}. \\

Finally, we focus on $\mu^{(4)}(n)$. We apply Lemma \ref{lem1} with $y=UV$, $v=0$ and $f = \mu_{\leqslant U}*\mu_{\leqslant V}$. Since $\|f\|_{\infty} \ll_{\varepsilon} (UV)^{\varepsilon} $ for any fixed $\varepsilon>0$, it follows that
\begin{align}\label{mu-4 sum}
\sum_{n\leqslant X }\mu^{(4)}(n)\e(n\alpha) \ll_{\varepsilon} (UV)^{\varepsilon}\bigg( UV+ \dfrac{X}{q} + q\bigg)(\log X).
\end{align}
We choose $U=V=\min\{ X^{2/5}, q, X/q \}$ and combine the estimates \eqref{mu-1 sum}, \eqref{mu-2 sum}, \eqref{mu-3 sum} and \eqref{mu-4 sum} to obtain
\begin{align*}
    \sum_{n\leqslant X }\mu(n)\e(n\alpha) \ll_{\varepsilon} X^{4/5+\varepsilon}+ \dfrac{X(\log X)^3}{q^{1/2}} + X^{1/2}q^{1/2}(\log X)^3.
\end{align*} 
This ends the proof. \qed
\subsection{Proof of Theorem \ref{theorem 2}}
The proof follows along the lines of the proof of Theorem \ref{theorem 1}. Again,  we can assume $q \leqslant X$. Let $U, V \geqslant 2$ be parameters to be chosen later on. Employing the Möbius inversion formula, we obtain
\begin{align*}
    \hat{\mu}(n)=\hat{\mu}_{ \leqslant U}(n) + (\hat{\mu}_{ > U}*\mu_{>V}*1)(n) + (f*\mu_{\leqslant V})(n) - (\hat{\mu}_{ \leqslant U}*\mu_{\leqslant V}*1)(n).
\end{align*}
Trivially, we have
\begin{align}\label{1st Estimate}
    \sum_{n\leqslant X }\hat{\mu}_{ \leqslant U}(n)\e(n\alpha) \ll  U \log U.
\end{align}
Next, by Lemma \ref{lem2}, similar to how we obtained \eqref{mu-3 sum}, we have
\begin{align}\label{2nd Estimate}
\sum_{n\leqslant X } (\hat{\mu}_{ > U}*\mu_{>V}*1)(n)\e(n\alpha) &\ll  X^{1/2}\bigg(\dfrac{X}{U}+\dfrac{X}{V}+q + \dfrac{X}{q} \bigg)^{1/2} (\log X)^{5/2}\notag \\
&\ll \bigg(X^{1/2}q^{1/2}+\dfrac{X}{U^{1/2}}+\dfrac{X}{V^{1/2}}  + \dfrac{X}{q^{1/2}}\bigg)(\log X)^{5/2}.
\end{align}
Observing that $\|f \cdot 1_{n\leqslant X} \|_2 \ll X^{1/4}$, again, an application of Lemma \ref{lem2} shows that
\begin{align}\label{3rd Estimate}
\sum_{n\leqslant X }(f*\mu_{\leqslant V})(n)\e(n\alpha) &\ll \bigg(q + X + V + \dfrac{X}{q}\bigg)^{1/2}X^{1/4}V^{1/4}(\log X)^{3/2}\notag \\
&\ll \bigg(X^{1/4}V^{1/4}q^{1/2}+X^{3/4}V^{1/4}\bigg)(\log X)^{3/2}.
\end{align}
Arguing as in the proof of \eqref{mu-4 sum} for the case of the M\"{o}bius function, we have
\begin{align}\label{4th Estimate}
\sum_{n\leqslant X }(\hat{\mu}_{ \leqslant U}*\mu_{\leqslant V}*1)(n)\e(n\alpha) \ll_{\varepsilon} (UV)^{\varepsilon}\bigg( UV+ \dfrac{X}{q} + q\bigg)(\log X),
\end{align}
for any fixed $\varepsilon>0$. Finally, we choose $U=V=\min\{ X^{1/3}, q, X/q \}$ and combine \eqref{1st Estimate}, \eqref{2nd Estimate}, \eqref{3rd Estimate} and \eqref{4th Estimate} to obtain
\begin{align*}
\sum_{n\leqslant X }\hat{\mu}(n)\e(n\alpha) \ll \bigg(X^{5/6} + q^{1/2}X^{1/2} + \dfrac{X}{q^{1/2}}\bigg)(\log X)^{5/2},
\end{align*}
which is the desired result. \qed

\subsection{Proof of Theorem \ref{theorem 3}}
Similar to the previous proofs, we can assume $q \leqslant X$. Let $U, V \geqslant 2$ be parameters to be chosen later on. We start with the following identity
\begin{align*}
\tilde{\mu}_2(n) &= \tilde{\mu}_2(n)\cdot 1_{n\leqslant U} + (\tilde{\mu}_2\cdot1_{n>U}*\mu_{>V}*1)(n) \notag \\
&\quad + (\mu^2*\mu_{\leqslant V})(n)-(\tilde{\mu}_2 \cdot 1_{n \leqslant U}*\mu_{\leqslant V}*1)(n).
\end{align*}
Arguing as in the previous proofs, we can show that
\begin{align}
\sum_{n\leqslant X }\tilde{\mu}_2(n)\cdot 1_{n \leqslant U} \e(n\alpha) &\ll U \log U, \label{Thm 3 : Step 2A}\\
\sum_{n\leqslant X }(\tilde{\mu}_2\cdot1_{n>U}*\mu_{ >V}*1)(n)\e(n\alpha) &\ll \bigg(X^{1/2}q^{1/2} + \dfrac{X}{U^{1/2}}+\dfrac{X}{V^{1/2}} +\dfrac{X}{q^{1/2}}\bigg) (\log X)^3,\label{Thm 3 : Step 2B} \\
\textrm{and} \quad \sum_{n\leqslant X }(\mu_{ \leqslant V}*\tilde{\mu}_2 \cdot 1_{n \leqslant U}*1)(n) e(n\alpha) &\ll_{\varepsilon} (UV)^{\varepsilon}\bigg( UV+ \dfrac{X}{q} + q\bigg)(\log X), \label{Thm 3 : Step 2C}
\end{align}
for any fixed $\epsilon>0$. It remains to estimate the sum
\begin{align*}
    \sum_{n\leqslant X }(\mu^2*\mu_{\leqslant V})(n)\e(n\alpha).
\end{align*}
A treatment like $\eqref{3rd Estimate}$ is not enough in this case because $\|\mu^2 \cdot 1_{n\leqslant X} \|_2 \ll X^{1/2}$. Therefore, we proceed in a different way. Using the relation $\mu^2(m)=\sum_{ab^2=n}\mu(b)$, we can write
\begin{align}
     \sum_{n\leqslant X }(\mu^2*\mu_{\leqslant V})(n)\e(n\alpha) &=\sum_{\substack{mb^2\leqslant X \\ m\leqslant V}}\mu(b)\mu(m)\sum_{a\leqslant X /(mb^2)}\e(amb^2\alpha) \notag \\
     &\ll \sum_{b\leqslant X ^{1/2}}\sum_{m\leqslant V}\min\bigg\{\dfrac{X}{b^2m}, \dfrac{1}{\|b^2m\alpha\|}\bigg\} \label{Thm 3 : Step 3}.
\end{align}
We need to break the range of $b$ in \eqref{Thm 3 : Step 3} into two parts: for large $b$'s, we will use the bound from Lemma \ref{lem: Mik}, and for small $b$'s, we will use the Cauchy-Schwarz inequality along with Lemma \ref{lem1}. Indeed, we write
\begin{align*}
     \sum_{b\leqslant X ^{1/2}}\sum_{m\leqslant V}\min\bigg\{\dfrac{X}{b^2m}, \dfrac{1}{\|b^2m\alpha\|}\bigg\} =  \sum_{b\leqslant B}\sum_{m\leqslant V}\min\bigg\{\dfrac{X}{b^2m}, \dfrac{1}{\|b^2m\alpha\|}\bigg\} +  \sum_{B<b\leqslant X ^{1/2}}\sum_{m\leqslant V}\min\bigg\{\dfrac{X}{b^2m}, \dfrac{1}{\|b^2m\alpha\|}\bigg\}.
\end{align*}
Then, by applying the Cauchy-Schwarz inequality,
\begin{align}
    \sum_{b\leqslant B}\sum_{m\leqslant V}\min\bigg\{\dfrac{X}{b^2n}, \dfrac{1}{\|b^2n\alpha\|}\bigg\} &\leqslant\sum_{m\leqslant B^2V}\tau_3(m)\min\bigg\{\dfrac{X}{m}, \dfrac{1}{\|m\alpha\|}\bigg\} \notag \\
    &\ll X^{1/2}\bigg(B^2V + q + \dfrac{X}{q}\bigg)^{1/2} (\log X)^6 \notag\\
    &\ll \bigg(X^{1/2}BV^{1/2} + X^{1/2}q^{1/2} + \dfrac{X}{q^{1/2}}\bigg)(\log X)^6. \label{Thm 3 : Step 4}
\end{align}
Next, by Lemma \ref{lem: Mik} dyadically, we have
\begin{align}
  \sum_{B<b\leqslant X ^{1/2}}\sum_{m\leqslant V}\min\bigg\{\dfrac{X}{b^2m}, \dfrac{1}{\|b^2m\alpha\|}\bigg\} \ll \bigg(X^{1/2}V + \dfrac{X}{q^{1/4}B} + \dfrac{X}{B^{5/4}} + \dfrac{X^{3/4}q^{1/4}}{B}\bigg)(\log X)^{10}. \label{Thm 3 : Step 5} 
\end{align}
Choose $B=(X/V)^{2/9}$ and combining \eqref{Thm 3 : Step 4} and \eqref{Thm 3 : Step 5}, we obtain
\begin{align}
   \sum_{b\leqslant X ^{1/2}}\sum_{m\leqslant V}\min\bigg\{\dfrac{X}{b^2m}, \dfrac{1}{\|b^2m\alpha\|}\bigg\}\ll \bigg(X^{13/18}V^{5/18} + X^{1/2}V + \dfrac{X}{q^{1/4}} + X^{3/4}q^{1/4}\bigg)(\log X)^{10}. \label{Thm 3 : Step 6}  
\end{align}
Finally, we choose $U=V=\min\{ X^{5/14}, q, X/q \}$ and combine \eqref{Thm 3 : Step 2A}, \eqref{Thm 3 : Step 2B}, \eqref{Thm 3 : Step 2C} and \eqref{Thm 3 : Step 6} to obtain
\begin{align*}
    \sum_{n\leqslant X }\tilde{\mu}_2(n)\e(n\alpha)\ll \bigg(X^{23/28} + \dfrac{X}{q^{1/4}} + X^{3/4}q^{1/4}\bigg)(\log X)^{10},
\end{align*}
as it was to be shown.
\qed

\section{Proof of Theorem \ref{Mobius-k convolutions}} \label{sec:secondexpsumproof}
In order to prove Theorem \ref{Mobius-k convolutions}, we will require the following lemma. 

\begin{lemma}\label{Gamma-Mobius}
Let $\alpha\in \mathbb{R}$ and consider a reduced fraction $a/q$ such that $|\alpha-a/q|\leqslant\gamma/q^2$ for some $\gamma \geqslant 1$. Then for $X \geqslant 2$ and for any fixed $\varepsilon>0$, we have
\begin{align}\label{Mu-Gamma}
S_{\mu}(X, \alpha) \ll_{\varepsilon} \gamma \bigg(X^{4/5+\varepsilon}+\dfrac{X(\log X)^3}{q^{1/2}} + X^{1/2}q^{1/2} (\log X)^3\bigg ),
\end{align}
where the implied constant in \eqref{Mu-Gamma} depends at most on $\varepsilon$.
\end{lemma}
\begin{proof}
    The proof follows exactly along the lines of the proof of Theorem \ref{theorem 1}, except that we apply Lemmas \ref{lem5} and \ref{lem6} instead of Lemmas \ref{lem1} and \ref{lem2} respectively.
\end{proof}
We explicitly show the proof of Theorem \ref{Mobius-k convolutions} in the case when $k=2$. The argument then follows similarly for higher values of $k$ using induction. We can assume $q \leqslant X$, otherwise, the conclusion is trivial. We break our proof into several steps.

\subsection{Initial step} Let 
\begin{align*}
    S_{\mu_2}(X, \alpha):=\sum_{n\leqslant X }(\mu*\mu)(n)\e(n\alpha).
\end{align*}
Then, for some parameters $2\leqslant M , N\leqslant X $ to be chosen later, we can write
\begin{align} 
S_{\mu_2}(X, \alpha)=&\: \sum_{\substack{mn\leqslant X \\ m> M, n>N}}\mu(m)\mu(n)\e(mn\alpha) + \sum_{m\leqslant M }\mu(m)\sum_{n\leqslant X /m}\mu(n)\e(mn\alpha)\notag \\
\notag &+\sum_{n\leqslant N}\mu(n)\sum_{m\leqslant X /n}\mu(m)\e(mn\alpha)-\sum_{\substack{m\leqslant M\\ n\leqslant N }}\mu(m)\mu(n)\e(mn\alpha)\\
=&\:  S_{\mu_2}^{(1)}(X, \alpha)+S_{\mu_2}^{(2)}(X, \alpha) + S_{\mu_2}^{(3)}(X, \alpha)-S_{\mu_2}^{(4)}(X, \alpha) \label{Dividing Sums} ,
\end{align}
say. We now turn our attention to each of these terms.
\subsection{Estimation of $S_{\mu_2}^{(1)}(X, \alpha)$ and $S_{\mu_2}^{(4)}(X, \alpha)$.} Applying Lemma \ref{lem2} after the dyadic decomposition following the proof of Theorem \ref{theorem 1}, we deduce that
\begin{align}\notag 
S_{\mu_2}^{(1)}(X, \alpha)&\ll X^{1/2}\bigg(q+\dfrac{X}{M} + \dfrac{X}{N} + \dfrac{X}{q}\bigg)^{1/2}(\log X)^2\\
&\ll \bigg(X^{1/2}q^{1/2}+\dfrac{X}{M^{1/2}} + \dfrac{X}{N^{1/2}} + \dfrac{X}{q^{1/2}}\bigg)(\log X)^2 \label{S_1 bound}.
\end{align}
Trivially, we have
\begin{align}
S_{\mu_2}^{(4)}(X, \alpha) \ll  MN. \label{S_4 bound}
\end{align}
\subsection{Estimation of $S_{\mu_2}^{(2)}(X, \alpha)$ and $S_{\mu_2}^{(3)}(X, \alpha)$.}  Fix $\varepsilon>0$. Writing $\alpha = a/q + \beta$, one has
\[m\alpha = \dfrac{ma}{q} + m\beta = \dfrac{ma/(m, q)}{q/(m, q)} + m\beta \quad \text{and}\quad |m\beta|\leqslant\dfrac{m/(m, q)^2}{q^2/(m, q)^2} \leqslant\dfrac{m/(m, q)}{q^2/(m, q)^2}.\]
Therefore, using Lemma \eqref{Gamma-Mobius}, we obtain
\begin{align}
    \notag  S_{\mu_2}^{(2)}(X, \alpha) & \ll_{\varepsilon} \sum_{m\leqslant M}\dfrac{m}{(m, q)}\bigg\{\bigg(\dfrac{X}{m}\bigg)^{4/5+\varepsilon} + \dfrac{X(m, q)^{1/2}(\log X)^3}{mq^{1/2}} + \bigg(\dfrac{X}{m}\bigg)^{1/2}\dfrac{q^{1/2}(\log X)^3}{(q, m)^{1/2}}\bigg \} \\
    &\ll_{\varepsilon} X^{4/5+\varepsilon}M^{6/5-\varepsilon} + \dfrac{XM(\log X)^3}{q^{1/2}} + X^{1/2}q^{1/2}M^{3/2}(\log X)^3. \label{S_2 Bound}
\end{align}
Similarly, we can show that
\begin{align}
    S_{\mu_2}^{(3)}(X, \alpha)\ll_{\varepsilon}  X^{4/5+\varepsilon}N^{6/5-\varepsilon} + \dfrac{XN(\log X)^3}{q^{1/2}} + X^{1/2}q^{1/2}N^{3/2}(\log X)^3. \label{S_3 bound}
\end{align}
\subsection{Parameter optimization}
Substituting the estimates \eqref{S_1 bound}, \eqref{S_2 Bound}, \eqref{S_3 bound} and \eqref{S_4 bound} in \eqref{Dividing Sums} and choosing $M=N$, we deduce that
\begin{equation}
    S_{\mu_2}(X, \alpha) \ll_{\varepsilon}  X^{4/5+\varepsilon}M^{6/5-\varepsilon}+\bigg(\dfrac{X}{M^{1/2}} + X^{1/2}q^{1/2} + \dfrac{XM}{q^{1/2}} + X^{1/2}q^{1/2}M^{3/2}\bigg)(\log X)^3. \notag
\end{equation}
We are now left with the crucial step,  to choose our parameter $M$ optimally. Let $M=y^5$, $A= \sqrt{q}$, $B = X^{1/5}$ and $C  = \sqrt{{X}/{q}}$.  Then it suffices to find an upper bound for the expression
\begin{align}\label{Optimization Step}
    H(y):= \max \bigg \{ \dfrac{1}{y^{5/2}}, \dfrac{y^5}{A}, \dfrac{y^6}{B}, \dfrac{y^{15/2}}{C}  \bigg \}. 
\end{align} 
Note that as a function of $y$, only $1/y^{5/2}$ is decreasing. Hence, the maximum in \eqref{Optimization Step} is attained when $1/y^3$ is equal to one of the other three functions. Consider the case when 
\begin{align} \label{Case 1 Condition}
   \dfrac{y^5}{A} \geqslant \max \bigg \{ \dfrac{y^6}{B}, \dfrac{y^{15/2}}{C}  \bigg \},
\end{align} 
which occurs when $1 \leqslant q \leqslant X^{6/17}$. Then the maximum in \eqref{Optimization Step} is attained when $y = A^{2/15}$ 
and therefore we obtain
\begin{align} \label{Case 1 Bound}
    S_{\mu_2}(X, \alpha) \ll  \dfrac{X(\log X)^3}{q^{1/6}}.
\end{align}
The treatment for the other cases is similar. In particular, when $y^6/B$ is the maximum, one requires 
$X^{6/17}\leqslant q \leqslant X^{9/17}$ and in this range, we have
\begin{align} \label{Case 2 Bound}
    S_{\mu_2}(X, \alpha) \ll_{\varepsilon}  X^{16/17+\varepsilon}.
\end{align}
Finally, when $X^{9/17}\leqslant q \leqslant X$, then $y^{15/2}/C$ is the maximum and we arrive at 
\begin{align} \label{Case 3 Bound}
    S_{\mu_2}(X, \alpha) \ll  X^{7/8}q^{1/8}(\log X)^3.
\end{align}
Combining the estimates \eqref{Case 1 Bound}, \eqref{Case 2 Bound} and \eqref{Case 3 Bound}, we conclude that
\[
\sum_{n\leqslant X }(\mu*\mu)(n)\e(n\alpha) \ll_{\varepsilon} X^{16/17+\varepsilon}+\bigg(\dfrac{X}{q^{1/6}} + X^{7/8}q^{1/8}\bigg)(\log X)^{3}. 
\]
This completes the proof for the case when $k=2$.

\subsection{Induction step} Fix $\varepsilon>0$. Assume that for $X \geqslant 2$,
\begin{align*}
   \sum_{n\leqslant X }\mu_k(n)\e(n\alpha) \ll_{\varepsilon} X^{a_k+\varepsilon}+ \dfrac{X(\log X)^{k^2}}{q^{b_k}} + X^{c_k}q^{1-c_k}(\log X)^{k^2}, 
\end{align*}
for some $k=\ell \geqslant 2$ and we want to prove our desired result when $k=\ell+1$. Similar to Lemma \ref{Gamma-Mobius}, one can show that if $\alpha\in \mathbb{R}$ such that $|\alpha-a/q|\leqslant \gamma/q^2$ for some $(a, q)=1$ and $\gamma \geqslant 1$, then
\begin{align}\label{With gamma}
   \sum_{n\leqslant X }\mu_\ell(n)\e(n\alpha) \ll_{\varepsilon} \gamma \bigg(X^{a_{\ell}+\varepsilon}+ \dfrac{X(\log X)^{\ell^2}}{q^{b_{\ell}}} + X^{c_{\ell}}q^{1-c_{\ell}}(\log X)^{\ell^2} \bigg), 
\end{align}
where the implied constant in \eqref{With gamma} depends at most on $\varepsilon$. Equipped with \eqref{With gamma}, for some parameters $2\leqslant M, N\leqslant X$ to be chosen later, we can write
\begin{align} 
S_{\mu_{\ell+1}}(X, \alpha)=&\: \sum_{\substack{mn\leqslant X \\ m> M, n>N}}\mu(m)\mu_\ell (n)\e(mn\alpha) + \sum_{m\leqslant M}\mu(m)\sum_{n\leqslant X/m}\mu_\ell (n)\e(mn\alpha)\notag \\
\notag &+\sum_{n\leqslant N}\mu_\ell (n)\sum_{m\leqslant X/n}\mu(m)\e(mn\alpha)-\sum_{\substack{m\leqslant M\\ n\leqslant N}}\mu(m)\mu _\ell (n)\e(mn\alpha)\\
=&\:  S_{\mu _{\ell+1} }^{(1)}(X, \alpha)+S_{\mu _{\ell+1} }^{(2)}(X, \alpha) + S_{\mu _{\ell+1} }^{(3)}(X, \alpha)-S_{\mu _{\ell+1} }^{(4)}(X, \alpha) \label{Dividing Sums I}.
\end{align}
Again, applying Lemma \ref{lem2}, we obtain
\begin{align} 
S_{\mu _{\ell+1} }^{(1)}(X, \alpha) \ll \bigg(X^{1/2}q^{1/2}+\dfrac{X}{M^{1/2}} + \dfrac{X}{N^{1/2}} + \dfrac{X}{q^{1/2}}\bigg)(\log X)^{(\ell+1)^2} \label{Bound 1}.
\end{align}
Trivially, we have
\begin{align}
S_{\mu _{\ell+1}}^{(4)}(X, \alpha) \ll  MN (\log N)^{\ell-1}. \label{Bound 4}
\end{align}
Choosing $\gamma = m/(m,q)$ and applying \eqref{With gamma}, we arrive at
\begin{align}
    S_{\mu _{\ell+1} }^{(2)}(X, \alpha) 
    \ll_{\varepsilon} X^{a_\ell+\varepsilon}M^{2-a_\ell-\varepsilon}+ \dfrac{XM(\log X)^{(\ell+1)^2}}{q^{b_\ell}} + X^{c _\ell}q^{1-c _\ell}M^{2-c_\ell}(\log X)^{(\ell+1)^2}. \label{Bound 2}
\end{align}
Note that Lemma \ref{Divisor Sum Powers} shows that
\[ \sum_{n \leqslant N } d _\ell^2(n) \ll N(\log N)^{\ell^2-1}.\]
Hence applying Lemma \eqref{Gamma-Mobius} and Cauchy-Schwartz inequality in conjunction, we can write
\begin{align}\label{Bound 3}
S_{\mu _{\ell+1} }^{(3)}(X, \alpha) &\ll_{\varepsilon} \sum_{n \leqslant N } d_\ell(n) \cdot \dfrac{n}{(n, q)}\bigg\{\bigg(\dfrac{X}{n}\bigg)^{4/5+\varepsilon} + \dfrac{X(n, q)^{1/2}(\log X)^3}{nq^{1/2}} + \bigg(\dfrac{X}{n}\bigg)^{1/2}\dfrac{q^{1/2}(\log X)^3}{(n, q)^{1/2}}\bigg\} \notag \\
&\ll_{\varepsilon}X^{4/5+\varepsilon}N^{6/5-\varepsilon}+ \dfrac{XN(\log X)^{(\ell+1)^2}}{q^{1/2}} + X^{1/2}q^{1/2}N^{3/2}(\log X)^{(\ell+1)^2}.
\end{align}
Substituting the estimates \eqref{Bound 1}, \eqref{Bound 4}, \eqref{Bound 2} and \eqref{Bound 3} in \eqref{Dividing Sums I} and choosing $M=N$, we have
\begin{equation*}
S_{\mu _{\ell+1} }(X, \alpha) \ll_{\varepsilon} X^{a _\ell+\varepsilon}M^{2-a _\ell-\varepsilon}+ \bigg(\dfrac{X}{M^{1/2}} +\dfrac{XM}{q^{b _\ell}} + X^{c _\ell}q^{1-c _\ell}M^{2-c _\ell} + X^{1/2}q^{1/2}M^{3/2}\bigg)(\log X)^{(\ell+1)^2}.
\end{equation*}
An optimization procedure similar to that for the case $k=2$ then shows that
\begin{align*}
S_{\mu _{\ell+1} }(X, \alpha) \ll_{\varepsilon}  X^{a _{\ell+1}+\varepsilon}+\dfrac{X(\log X)^{(\ell+1)^2}}{q^{b _{\ell+1}}} + X^{c _{\ell+1}}q^{1-c _{\ell+1}}(\log X)^{(\ell+1)^2},
\end{align*}
thereby completing the proof.
\qed
\section{Proof of Theorem \ref{Partition}} \label{sec:partitions}
We now set up the structure of the proof of Theorem \ref{Partition}. Recall from \eqref{2 Convolution Mobius} and \eqref{eq 1.10} that
\begin{align*}
    \Psi_{\mu*\mu}(z) &= \sum_{n=1}^\infty p_{\mu*\mu}(n) z^n = \prod_{n=1}^\infty (1-z^n)^{-(\mu*\mu)(n)}, \quad \lvert z \rvert <1, \\
    \Phi_{\mu*\mu}(z) &= \sum_{j=1}^\infty \sum_{n=1}^\infty \frac{(\mu*\mu)(n)}{j} z^{jn}, \quad \lvert z \rvert <1,
\end{align*}
and $\Psi_{\mu*\mu}(z) = \exp(\Phi_{\mu*\mu}(z))$. The strategy to prove Theorem \ref{Partition} is to employ the Hardy-Littlewood circle method. An application of Cauchy's theorem yields
\begin{align}\label{Cauchy}
    p_{\mu*\mu}(n) = \rho^{-n} \int_0^1 \exp(\Phi_{\mu*\mu}(\rho \e(\theta))-2\pi i n \theta) d\theta,
\end{align}
where $0< \rho < 1$. We write $\rho = \exp(-1/X)$, for some parameter $X>0$. As we shall see, for each value of $n$, we will choose $X = X(n)$ and therefore, $\rho=\rho(n)$ in \eqref{Cauchy} to be such that as $n \to \infty$, $X \to \infty$ and $\rho \to 1$. Therefore, in what follows, we can assume $X$ to be sufficiently large. Invoking the periodicity of the integrand in \eqref{Cauchy}, we can replace the unit interval by the interval $\mathcal{U} = [-X^{-1} (\log X)^A ,1-X^{-1} (\log X)^A)$, where $A>0$ is a sufficiently large constant. This choice $\mathcal{U}$ keeps the region around the origin from being split into two regions.\\

Now we are ready to set up the major and minor arcs. Let $a \in \mathbb{N}\cup \{0\}$ and $q \in \mathbb{N}$ such that $0 \leqslant a < q$ and $(a,q)=1$. Choose
\begin{align} \label{eq:def_delta_q}
    \delta_q = q^{-1}X^{-1}(\log X)^A \quad \textnormal{and} \quad Q = (\log X)^A,
\end{align} and define
\begin{align}
    \mathfrak{M}(q,a) = \bigg\{\theta \in \mathcal{U}: \bigg|\theta-\frac{a}{q}\bigg| \leqslant \delta_q \bigg\}.
\end{align}
The major arcs $\mathfrak{M}$ and the minor arcs $\mathfrak{m}$ are
\begin{align}\label{Major and Minor Arcs}
    \mathfrak{M} = \bigcup_{\substack{1 \leqslant a \leqslant q \leqslant Q \\ (a,q)=1}} \mathfrak{M}(q,a) \quad \textnormal{and} \quad \mathfrak{m} =  \mathcal{U} \backslash \mathfrak{M}.
\end{align}
We will show that if $\theta \in \mathfrak{M}$, then
\begin{align}\label{eq 1}
        \Phi_{\mu * \mu}(\rho \e (\theta)) \ll_A \frac{X}{(\log X)^{A}}.
\end{align}
On the other hand, if $\theta \in \mathfrak{m}$, then one has 
\begin{align}\label{eq 2}
        \Phi_{\mu * \mu}(\rho \e (\theta)) \ll_A \frac{X}{(\log X)^{A/9}}.
\end{align}  
With \eqref{eq 1} and \eqref{eq 2} in mind, let us first prove Theorem \ref{Partition}.

\begin{proof}[Proof of Theorem \textnormal{\ref{Partition}}] We rewrite \eqref{Cauchy} as
\begin{align}
    p_{\mu*\mu}(n) = \rho^{-n}\bigg(\int_{\mathfrak{M}}  + \int_{\mathfrak{m}}\bigg) \exp(\Phi_{\mu*\mu}(\rho \e (\theta))-2\pi i n \theta) d\theta.
\end{align}
First, we consider the integral over the minor arcs . Let $X = \sqrt{n} (\log n)^{A/18}$. Since $\rho =\exp(-1/X)$, using \eqref{eq 2}, we can write
\begin{align}
\rho^{-n} \int_{\mathfrak{m}} \exp(\Phi_{\mu*\mu}(\rho \e (\theta))-2\pi i n \theta) d\theta \ll \exp \bigg(\frac{n}{X}+\frac{X}{(\log X)^{A/9}} \bigg) \ll \exp \bigg(\frac{\sqrt{n}}{(\log n)^{A/20}} \bigg).
\end{align}
A concomitant argument holds for the major arcs. Therefore, choosing $A$ sufficiently large, we obtain
\[
\log p_{\mu*\mu}(n) \ll_B \frac{\sqrt{n}}{(\log n)^{B}},
\]
for any fixed $B>0$.
\end{proof}

Therefore, it suffices to establish \eqref{eq 1} and \eqref{eq 2}, which we do in Sections \ref{sec:nonprincipal} and \ref{sec:minorarcs} respectively.

\section{Major arcs} \label{sec:nonprincipal}
We begin our discussion on the major arcs by estimating M\"{o}bius convolutions in arithmetic progressions. The key ingredient required for the following result is Lemma \ref{Davenport}.
\begin{lemma}\label{Davenport Lemma for Convolutions}
Let $X \geqslant 2$. Suppose $B>0$ is fixed. Let $r, q\geqslant 1$ be positive integers with $q\leqslant(\log X)^B$. Then there exists a constant $c_B>0$ depending only on $B$ such that
\begin{align*}
    S(X) := \sum_{\substack{n \leqslant X \\ n \equiv r \modu q}} (\mu * \mu)(n) \ll_B X\exp(-c_B\sqrt{\log X}).
\end{align*}
\end{lemma}
\begin{proof}
By Dirichlet's hyperbola method, we break $S(X)$ into three sub-sums, i.e., $S(X) = S_1(X)+S_2(X) - S_3(X)$, where
\begin{align}
\notag S_1(X) &= \sum_{n_1 \leqslant \sqrt{X}} \mu(n_1) \sum_{\substack{n_2 \leqslant X/n_1 \\ n_1 n_2 \equiv r \modu q}} \mu(n_2),\\
\notag  S_2(X) &= \sum_{n_2 \leqslant \sqrt{X}} \mu(n_2) \sum_{\substack{n_1 \leqslant X/n_2 \\ n_1 n_2 \equiv r \modu q}} \mu(n_2),\\
\notag S_3(X) &= \sum_{n_1 \leqslant \sqrt{X}} \mu(n_1) \notag \sum_{\substack{n_2 \leqslant \sqrt{X} \\ n_1 n_2 \equiv r \modu q}} \mu(n_2).
\end{align}
We treat $S_3(X)$ first. Consider the inner sum over $n_2$.  Note that if $n_1 n_2 \equiv r \modu q$, then $n_1(n_2 + q) \equiv r \modu q$. With this in mind, we write 
\begin{align}\label{Def : B}
\mathcal{B}:=\big\{x\modu q\colon n_1x \equiv r \modu q\big\}.
\end{align}
Then using Lemma \ref{Davenport}, it follows that
\begin{align}\label{A3 Estimate}
\sum_{\substack{n_2 \leqslant \sqrt{X} \\ n_1 n_2 \equiv r \modu q}} \mu(n_2) = \sum_{b \in \mathcal{B}} \sum_{\substack{n_2 \equiv b \modu q \\ n_2 \leqslant \sqrt{X}}} \mu(n_2) \ll_B \operatorname{card}(\mathcal{B}) \sqrt{X} \exp(-c_1 \sqrt{\log X}),
\end{align}
for some constant $c_1=c_1(B)>0$. The bound in \eqref{A3 Estimate} holds for all $q \leqslant (\log X)^B$. Therefore, by trivially bounding the sum over $n_1$, we obtain
\begin{align}\label{A3 Estimate Final}
S_3(X) \ll_B \bigg(\sum_{n_1 \leqslant \sqrt{X}} |\mu(n_1)|\bigg) \operatorname{card}(\mathcal{B}) \sqrt{X} \exp(-c_1 \sqrt{\log X}) \ll_B X \exp(-c_2 \sqrt{\log X}),
\end{align}
for some constant $c_2=c_2(B)>0$.

Note that $S_1(X)=S_2(X)$ by symmetry. Therefore, it is enough to estimate the sum $S_1(X)$. We analyze the inner sum over $n_2$ first. Again applying Lemma \ref{Davenport}, we obtain
\begin{align*}
\sum_{\substack{n_2 \leqslant X/n_1 \\ n_1 n_2 \equiv r \modu q}} \mu(n_2) = \sum_{b \in \mathcal{B}} \sum_{\substack{n_2 \equiv b \modu q \\ n_2 \leqslant X/n_1}} \mu(n_2) \ll_B \operatorname{card}(\mathcal{B}) \frac{X}{n_1} \exp \bigg(-c_3 \sqrt{\log\frac{X}{n_1}}\bigg),
\end{align*}
for some constant $c_3=c_3(B)>0$, where $\mathcal{B}$ is as in \eqref{Def : B}. This leaves us with
\begin{align}
S_1(X) &\ll_B \sum_{n_1 \leqslant \sqrt{X}} \operatorname{card}(\mathcal{B}) \frac{X}{n_1} \exp \bigg(-c_3 \sqrt{\log\frac{X}{n_1}}\bigg) \ll_B X\exp\big(-c_4\sqrt{\log X}\big),\label{A1 Estimate}
\end{align}
for some constant $c_4=c_4(B)>0$. Combining \eqref{A3 Estimate Final} and \eqref{A1 Estimate}, we arrive at the conclusion.
\end{proof}
Applying Lemma \ref{Davenport Lemma for Convolutions}, one can quickly conclude the following.
\begin{lemma}\label{Partial Summation Lemma}
Let $X \geqslant 2$ and suppose $B>0$ is fixed. Let $\theta \in \mathbb{R}$ such that $\theta = a/q+\beta$ where $a,q \in \mathbb{Z}$ with $0 \leqslant a < q \leqslant(\log X)^B$, $(a,q)=1$ and $\lvert \beta \rvert \leqslant X ^{-1}(\log X)^B$. Then there exists a constant $c_B>0$ depending only on $B$ such that
\begin{align*}
    \sum_{n\leqslant X }(\mu*\mu)(n)\e(n\theta) \ll_B X\exp(-c_B\sqrt{\log X}).
\end{align*}
\end{lemma}
\begin{proof} By partial summation, we can write
\begin{align}\label{Partial Summation}
     \sum_{n\leqslant X }(\mu*\mu)(n)\e(n\theta)\ll_B (1 + |\beta|X)\max_{1\leqslant y\leqslant X }\bigg|\sum_{n\leqslant y}(\mu*\mu)(n)\e(na/q)\bigg|.
\end{align}
Consequently, it is enough to estimate the sum
\begin{align}\label{Breaking into arithmetic progressions}
    \sum_{n\leqslant y}(\mu*\mu)(n)\e(na/q)=\sum_{b=1}^q \e\bigg(\dfrac{ab}{q}\bigg)\sum_{\substack{n\leqslant y\\ n\equiv b\modu q}}(\mu*\mu)(n).
\end{align}
An immediate application of Lemma \ref{Davenport Lemma for Convolutions} shows that for some constant $c_1=c_1(B)>0$ depending only on $B$, one has
\begin{align}\label{AP Bound}
  \sum_{\substack{n\leqslant y\\ n \equiv b \modu q}}(\mu*\mu)(n)\ll_B y \exp(-c_1\sqrt{\log y}).
\end{align}
Combining \eqref{Partial Summation}, \eqref{Breaking into arithmetic progressions} and \eqref{AP Bound}, we conclude that
\begin{align*}
     \sum_{n\leqslant X }(\mu*\mu)(n)\e(n\theta)\ll_B (1+|\beta|X)qX\exp(-c_1\sqrt{\log X}) \ll_B X\exp(-c_B\sqrt{\log X}),
\end{align*}
for some constant $c_B>0$ depending only on $B$.
\end{proof}
Recall from \eqref{Major and Minor Arcs} that if $\theta \in \mathfrak{M}$, we can write $\theta=a/q+\beta$ with $q\leqslant(\log X)^A$ and $|\beta|\leqslant\delta_q \leqslant X ^{-1}(\log X)^A$, where $A>0$ is sufficiently large. Using Lemmas \ref{Davenport Lemma for Convolutions} and \ref{Partial Summation Lemma}, we can now prove \eqref{eq 1} for $\theta \in \mathfrak{M}$.

\begin{lemma}\label{Major Arcs Theorem}
Let $\theta \in \mathfrak{M}$ as defined in \eqref{Major and Minor Arcs}. Then for any fixed $A>0$,
\begin{align*}
        \Phi_{\mu * \mu}(\rho \e (\theta)) \ll_A \frac{X}{(\log X)^{A}}.
\end{align*}
\end{lemma}
\begin{proof} Since $\rho = \exp(-1/X)$, it follows from \eqref{eq 1.10} that
\begin{align*}
    \Phi_{\mu * \mu}(\rho \e(\theta)) = \sum_{j=1}^\infty\sum_{n=1}^\infty \frac{1}{j} (\mu * \mu)(n) e^{-nj / X} \e(j n \theta).
\end{align*}
We replace the exponentially damped term above by the integral
\begin{align*}
    e^{-nj / X} = \int_n^\infty jX^{-1}e^{-tj/X}\, dt,
\end{align*}
leaving us with
\begin{align*}
    \Phi_{\mu * \mu}(\rho \e(\theta)) = \sum_{j=1}^\infty \frac{1}{j} \int_1^\infty jX^{-1}e^{-tj/X} \sum_{n \leqslant t}  (\mu * \mu)(n) \e(j n \theta) \, dt.
\end{align*}
Suppose $J=(\log X)^{2A}$. Trivially, we obtain
\begin{align}\label{Integral Truncating}
\sum_{j>J}\frac{1}{j} \int_1^\infty jX^{-1}e^{-tj/X} &\sum_{n \leqslant t}  (\mu * \mu)(n) \e(j n \theta) \, dt \notag \\
   & \ll \sum_{j>J}^\infty \frac{1}{j} \int_1^\infty jX^{-1}e^{-tj/X} t \log t \, dt  \notag \\
   &\ll \sum_{j>J}^\infty \frac{X \log X}{j^2} \int_{j/X}^\infty ue^{-u} \log u  \, du \ll_A \frac{X}{(\log X)^{A}}.
\end{align}
Hence, in what follows, it suffices to work with the sum
\begin{align}\label{Truncated J}
\sum_{j=1}^J \frac{1}{j} \int_1^\infty jX^{-1}e^{-tj/X} \sum_{n \leqslant t}  (\mu * \mu)(n) \e(j n \theta) \, dt.
\end{align}
Writing $\theta=a/q+\beta$ with $q\leqslant(\log X)^A$ and $|\beta|\leqslant\delta_q \leqslant X ^{-1}(\log X)^A$, one has
\[ \theta _j := j\theta = \frac{aj/(j,q)}{q/(j,q)}+ j\beta = \frac{a _j}{q _j}+\beta _j,
\]
where $(a _j,q _j)=1$, $1 \leqslant q _j \leqslant q \leqslant(\log X)^{A} $ and $\beta _j \leqslant X ^{-1}(\log X)^{3A}$. We can now break the inner integral in \eqref{Truncated J} as
\begin{align}
  I &= \int_1^\infty jX^{-1}e^{-tj/X} \sum_{n \leqslant t}  (\mu * \mu)(n) \e(j n \theta) \, dt \notag \\
  &= \bigg(\int_1^{X \log X} + \int_{X\log X}^\infty\bigg) jX^{-1}e^{-tj/X} \sum_{n \leqslant t}  (\mu * \mu)(n) \e(j n \theta) \, dt = I_1+I_2 \notag, 
\end{align}
say. For $I_2$, one has
\begin{align}
    I_2 & \ll   \int_{X\log X}^\infty jX^{-1}e^{-tj/X} t \log t \, dt  \ll \frac{X \log X}{j} \int_{\log X}^\infty ue^{-u} \log u  \, du \ll_A \frac{X}{(\log X)^{3A}}. \label{I_2 Estimate}
\end{align}
For $I_1$, an application of Lemma \ref{Partial Summation Lemma} shows that
\begin{align}\label{I_1 Estimate}
    I_1 & \ll_A X \log X \exp(-C_A \sqrt{\log X})  \int_1^{X\log X} jX^{-1}e^{-tj/X} \, dt \notag \\
   & \ll_A X \log X \exp(-C_A \sqrt{\log X}) \int_{1}^{X\log X} e^{-u}   \, du \ll_A \frac{X}{(\log X)^{3A}},
\end{align}
where $C_A>0$ is a constant depending only on $A$. Putting together \eqref{I_2 Estimate} and \eqref{I_1 Estimate}, we obtain
\[ \sum_{j=1}^J \frac{1}{j} \int_1^\infty jX^{-1}e^{-tj/X} \sum_{n \leqslant t}  (\mu * \mu)(n) \e(j n \theta) \, dt \ll_A \frac{X}{(\log X)^{A}}.
\]
Combining this with \eqref{Integral Truncating}, we arrive at our desired conclusion.
\end{proof}
\begin{remark}
The above arguments and the statements of Lemmas \ref{Davenport Lemma for Convolutions}, \ref{Partial Summation Lemma} and \ref{Major Arcs Theorem} hold for a general $\mu_k$ with appropriate changes.
\end{remark}

\section{Minor Arcs} \label{sec:minorarcs}
In general, being able to show that the contribution of the minor arcs is of lower order than the principal arcs is the most challenging aspect of successfully deploying the circle method. In situations such as \cite{bmz, gafnipowers} which involved partitioning with respect to powers or powers in arithmetic progressions, one requires Weyl's inequality \cite[Lemma 2.4]{V1997} in conjunction with Hua's inequality \cite[Theorem 4.1]{V1997}. Adapting Weyl's inequality was also a key aspect for partitioning with respect to polynomials in \cite{dunnrobles}. Restricted partitions with respect to primes \cite{vaughanprimes} necessitated Vinogradov's inequality \cite[Chapter 3]{V1997}. However, Vinogradov's inequality was not sufficient for functions of primes. In \cite{gafniprimepowers}, Gafni employed a bound by Kawada and Wooley \cite{kawadawooley}, and in \cite{drzz}, a new bound altogether for exponential sums dealing with semiprimes was introduced and proved. This was also the case in \cite[$\mathsection$6]{brzz} where a bound for exponential sums twisted by generalized divisor functions had to be supplied. In our current situation, the exponential sums discussed in Section \ref{sec:introduction} and proved in Sections \ref{sec:proofsexpsums} and \ref{sec:secondexpsumproof} will serve as the main technology to bound the contributions from the minor arcs. It is worth remarking that interesting exponential sums for minor arcs can also be found in \cite{kumchevalmostprimes}. \\

We begin by recording the immediate corollary from Theorem \ref{Mobius-k convolutions} for the case when $k=2$.
\begin{corollary} \label{theorem 4}
Let $\alpha\in \mathbb{R}$ and consider a reduced fraction $a/q$ such that $|\alpha-a/q|\leqslant1/q^2$. Then for $X \geqslant 2$ and any fixed $\varepsilon>0$, we have
\begin{align*}
\sum_{n\leqslant X }(\mu*\mu)(n)\e(n\alpha) \ll_{\varepsilon} X^{16/17+\varepsilon}+\bigg(\dfrac{X}{q^{1/6}} + X^{7/8}q^{1/8}\bigg)(\log X)^{3}. 
\end{align*}
\end{corollary}
Now, we are ready to prove \eqref{eq 2} for $\theta \in \mathfrak{m}$.
\begin{lemma} \label{lem:minor}
Let $\theta \in \mathfrak{m}$ as defined in \eqref{Major and Minor Arcs}. Then for any fixed $A>0$, 
\begin{align*}
        \Phi_{\mu * \mu}(\rho \e (\theta)) \ll_A \frac{X}{(\log X)^{A/9}},
\end{align*} 
\end{lemma} 
\begin{proof}
As in the proof of Lemma \ref{Major Arcs Theorem}, it suffices to only consider the sum 
\[ \sum_{j=1}^J \frac{1}{j} \int_1^\infty jX^{-1}e^{-tj/X} \sum_{n \leqslant t}  (\mu * \mu)(n) \e(j n \theta) \, dt,
\]
where $J=(\log X)^{2A}$. For each $j \leqslant J$, we employ Dirichlet's theorem (see \cite[Lemma 2.1]{V1997}) to choose $a \in \mathbb{Z}$ and $q \in \mathbb{N}$ with $(a,q)=1$ such that 
\begin{align*}
\bigg| j \theta - \frac{a}{q} \bigg| \leqslant q^{-1} X^{-1} (\log X)^A \quad \textnormal{and} \quad q < X(\log X)^{-A}.
\end{align*}
We now set $a _j:= a/(a,j)$ and $q _j:=jq/(a,j)$.
The definition of $\delta_q$ in \eqref{eq:def_delta_q} imply that
\begin{align*}
\bigg|  \alpha - \frac{a _j}{q _j} \bigg| \leqslant \delta_{q _j}.
\end{align*}
Since $\alpha \in \mathfrak{m}$, it follows that $q _j > Q$, where $Q=(\log X)^A$. Now, applying Corollary \ref{theorem 4}, we obtain 
\begin{align} \label{auxminorarcs01}
\sum_{n \leqslant t}  (\mu * \mu)(n) \e(j n \theta) \ll_{\varepsilon} t^{16/17+\varepsilon}+\dfrac{t(\log t)^{3}}{q^{1/6}} + t^{7/8}q^{1/8}(\log t)^{3}.
\end{align}
Integrating each of the three terms on the right hand side of \eqref{auxminorarcs01} yields
\begin{align*}
E_1(X,j) &= \int_1^\infty t^{16/17+\varepsilon} j X^{-1} e^{-tj/X} \, dt \ll_{\varepsilon} \bigg(\frac{X}{j}\bigg)^{16/17+\varepsilon}, \\
E_2(X,j) &= \int_1^\infty \frac{t}{q^{1/6}} (\log t)^3 j X^{-1} e^{-tj/X} \, dt \ll \frac{X}{{j{q^{1/6}}}}\left( {\log X} \right)^3, \\
E_3(X,j) &= \int_1^\infty t^{\frac{7}{8}}q^{\frac{1}{8}} (\log t)^{3} j X^{-1} e^{-tj/X} \, dt \ll {\left( {\frac{X}{j}} \right)^{\frac{7}{8}}}{q^{1/8}}\left( {\log X} \right)^3.
\end{align*}
Note that $jq \geqslant q _j > (\log X)^{A}$, and by our choice, $ q < X(\log X)^{-A} $. Therefore, we conclude that
\begin{align*}
\sum_{j=1}^J \frac{1}{j} \int_1^\infty jX^{-1}e^{-tj/X} \sum_{n \leqslant t}  (\mu * \mu)(n) \e(j n \theta) \, dt \ll \sum_{j=1}^J \frac{1}{j} (E_1+E_2+E_3) \ll \frac{X}{(\log X)^{A/9}},
\end{align*}   
which is the bound that we wanted to show.
\end{proof}
\begin{remark}
The statement of Lemma 7.2 holds more generally for $\mu_k$ by applying Theorem \ref{Mobius-k convolutions} appropriately in place of Corollary \ref{theorem 4}.  
\end{remark}
\section{Proof of Theorem \ref{lem:principalmajor}}\label{sec:principalarc}
In order to prove Theorem \ref{lem:principalmajor}, we need to investigate how $\Phi_{\mu*\mu}(z)$ behaves in the major arc and examine the effect that the non-trivial zeros of the Riemann zeta-function have in its behavior. We break our proof into several steps. 
\subsection*{Initial setup} Recall that
\begin{align} \label{eq:auxdefPhi}
    \Phi_{\mu * \mu}(\rho \e(\theta)) = \sum_{j=1}^\infty \sum_{n=1}^\infty \frac{(\mu * \mu)(n)}{j} \exp\bigg( -jn \bigg(\frac{1}{X}-2\pi i \theta\bigg)\bigg).
\end{align}
where $\rho =\exp(-1/X)$. We first apply the Cahen-Mellin formula and rewrite \eqref{eq:auxdefPhi} to obtain
\begin{align*}
    \Phi_{\mu * \mu}(\rho \e(\theta)) = \sum_{j=1}^\infty \sum_{n=1}^\infty \frac{(\mu * \mu)(n)}{j} \frac{1}{2 \pi i} \int_{(c)} \Gamma(s) j^{-s} n^{-s} \bigg(\frac{X}{1-2\pi i X \theta}\bigg)^s ds,
\end{align*}
for $c>0$. Interchanging the summations and the integration, for any real $c>1$, we have
\begin{align}
\Phi_{\mu * \mu}(\rho \e(\theta)) &= \frac{1}{2 \pi i} \int_{(c)} \Gamma(s) \bigg(\sum_{j=1}^\infty \frac{1}{j^{s+1}}\bigg) \bigg(\sum_{n=1}^\infty \frac{(\mu * \mu)(n)}{n^s} \bigg) \bigg(\frac{X}{1-2\pi i X \theta}\bigg)^s ds \nonumber \\
&= \frac{1}{2 \pi i} \int_{(c)} \Gamma(s) \zeta(s+1) \frac{1}{\zeta^2(s)} \bigg(\frac{X}{1-2\pi i X \theta}\bigg)^s ds. \label{Initial Step Sec 8}
\end{align}
\subsection*{Contour Shifting.} The path of integration in \ref{Initial Step Sec 8} is a vertical line with $\operatorname{Re}(s)=c >1$. We move the portion $|t| \leqslant T$ of this path to the left, replacing it by a rectangular path joining the points $c \pm i T$ and $-u \pm i T$, where $u=2N+1$, for some large $N \in \mathbb{N}$. Thus, the new path of integration is of the form $L=\bigcup_{i=1}^5 L_i$, with $L_1=(c-i \infty, c-i T],  L_2=[c-i T, -u-i T],  L_3=[-u-i T, -u+i T],  L_4=[-u+i T, c+i T]$ and $ L_5=[c+i T, c+i \infty)$.
With this change of the path of integration, an application of the Cauchy residue theorem shows that
\[
\Phi_{\mu * \mu}(\rho \e(\theta)) = M+\frac{1}{2 \pi i} \sum_{j=1}^5 I_j,
\]
where $M$ denotes the contribution of the residues at singularities of the integrand in the region enclosed by the two paths and $I_j$ denotes the integral over the path $L_j$. Our method will require an analysis of the $I_j$'s. Equipped with Lemma \ref{lem:RamachandraSankaranarayanan}, we will establish our results unconditionally. For similar techniques, the reader is referred to the works of Bartz \cite{Bartz}, Ramachandra and Sankaranarayanan \cite{RamachandraSankaranarayanan} and Inoue \cite{Inoue}.
\subsection*{Residue computations} Set 
\[y:=\frac{X}{1-2\pi i X\theta}.
\]
A singularity analysis regarding the integrand in \eqref{Initial Step Sec 8} shows the following. We have a double pole at $s=0$ coming from $\Gamma(s)\zeta(s+1)$ for which
\begin{align*}
\mathop{\operatorname{res}}\limits_{s=0} \Gamma(s) \zeta(s+1) \frac{1}{\zeta^2(s)} y^s = 4 \log y - 8 \log(2 \pi).
\end{align*}
The simple pole at $s=-1$ yields 
\begin{align*}
\mathop{\operatorname{res}}\limits_{s=-1} \Gamma(s) \zeta(s+1) \frac{1}{\zeta^2(s)} y^s = - \frac{\zeta(0)}{\zeta^2(-1)y} = \frac{72}{y}.
\end{align*}
Unlike previous cases in the literature \cite{bmz, brzz, drzz, dunnrobles, gafnipowers, gafniprimepowers}, we now capture residues from the trivial and non-trivial zeros of the Riemann zeta-function. We start with the non-trivial zeros. Those take place at $s = \rho$ and yield double poles for which
\begin{align*}
\mathop{\operatorname{res}}\limits_{s=\rho} \Gamma(s) \zeta(s+1) \frac{1}{\zeta^2(s)} y^s &= y^{\rho} \Gamma(\rho) \frac{\zeta(1+\rho)}{\zeta'(\rho)^2}\bigg(\log y + \psi(\rho) - \frac{\zeta''}{\zeta'}(\rho) + \frac{\zeta'}{\zeta}(1+\rho) \bigg),
\end{align*}
where we have used the simplicity of the zeros to alleviate the notation. For the trivial zeros, we have triple poles at $s=-2n$, for $n=1,2,3,\dots$, coming from $\Gamma(s)/\zeta^2(s)$. These residues can be written as
\begin{align*}
\mathop{\operatorname{res}}\limits_{s=-2n} \Gamma(s) \zeta(s+1) \frac{1}{\zeta^2(s)} y^s =  y^{-2 n}(\mathfrak{c}_1(n)(\log y)^2 + \mathfrak{c}_2(n)\log y + \mathfrak{c}_3(n)),
\end{align*}
where the coefficients are given by  
\begin{align}
\mathfrak{c}_1(n) &:= -\frac{B_{2 n}}{8 n^2 \Gamma (2 n) \zeta '(-2 n)^2}, \label{eq:defc1} \\
\mathfrak{c}_2(n) &:= \frac{B_{2 n} \zeta ''(-2 n)+\zeta '(-2 n) \left(2 n \zeta '(1-2 n)-B_{2 n} \psi (2n+1)\right)}{2 n (2 n)! \zeta '(-2 n)^3}, \label{eq:defc2} 
\end{align}
and
\begin{align} \label{eq:defc3}
    \mathfrak{c}_3(n) &:= \frac{1}{48 n^2 \Gamma (2 n) \zeta '(-2 n)^4} \nonumber \\
    &\quad \times \big[ B_{2 n} \big\{-9 \zeta ''(-2 n)^2-2 \left(3 \psi(2 n+1)^2-3 \psi'(2 n+1)+\pi
   ^2\right) \zeta '(-2 n)^2 \nonumber \\
   &\quad \quad \quad +4 \zeta '(-2 n) (\zeta '''(-2 n)+3 \psi(2 n+1)
   \zeta ''(-2 n))\big\} \nonumber \\
   &\quad \quad +12 n \zeta '(-2 n) \left(\zeta '(-2 n) \zeta ''(1-2 n)+2
   \zeta '(1-2 n) \left(\psi(2 n+1) \zeta '(-2 n)-\zeta ''(-2 n)\right)\right) \big].
\end{align}
Here $\psi(x) = \Gamma'/\Gamma(x)$ denotes the digamma function and $B_n$ denote the Bernoulli numbers. 

\subsection*{Estimating the vertical integral $I_3$} Using Stirling's formula, for $\operatorname{Re}(s) \leqslant-3/2$,
\begin{align}\label{Stirling Formula}
\Gamma(s)\zeta(s+1)\zeta(s) \ll (2 \pi)^{2\operatorname{Re}(s)}\lvert s \rvert^{-\operatorname{Re}(s)-1/2} e^{-\pi \lvert t \rvert/2}. 
\end{align}
Therefore, we have
\begin{align}\label{Starting Simplifications}
    I_3 = \int_{-u-iT}^{-u+iT} \Gamma(s)\zeta(s+1)\frac{1}{\zeta^2(s)}y^s \,ds &\ll \int_{-u-iT}^{-u+iT} \frac{(2 \pi)^{-2u}\lvert s \rvert^{u-1/2} e^{-\pi \lvert t \rvert/2}}{\zeta^3(s)}y^s \,ds \notag \\
    &\ll \frac{1}{(2\pi)^{2u}} \int_{-u-iT}^{-u+iT} \frac{(u^2+\lvert t \rvert^2 )^{u/2-1/4} e^{-\pi \lvert t \rvert/2}}{\zeta^3(s)}y^s \,ds.
\end{align}
We apply the functional equation for $\zeta(s)$ in \eqref{Starting Simplifications} alongside the observation that $\zeta^{-1}(1-s) \ll 1$. Also, following \cite{vaughansquares} (See Lemma 2.8),
\[
\lvert y \rvert^s \leqslant(X\Delta)^{-u}e^{-\lvert t \rvert(\Delta-\pi/2)}.\]
Hence, we obtain
\begin{align}\label{Starting Simplifications 2}
    I_3 &\ll \frac{2^{3u} \pi^{3(u-1)}}{(2\pi)^{2u} (X \Delta)^{u}} \int_{-u-iT}^{-u+iT} \frac{(u^2+\lvert t \rvert^2 )^{u/2-1/4} e^{-\lvert t \rvert \Delta}}{\sin^3(\pi s/2)\Gamma^3(1-s)} \,ds.
\end{align}
Next, we note that in our range of integration,
\begin{align*}
    \frac{1}{\Gamma(1-s)} \ll e^{(u+1) - (u+\frac{1}{2})\log (u+1) + \frac{1}{2}\pi |t|} \quad \textnormal{and} \quad \frac{1}{\sin(\pi s/2)} \ll e^{- \frac{1}{2}\pi |t|}.
\end{align*}
Inserting these two bounds in \eqref{Starting Simplifications 2}, we arrive at 
\begin{align}\label{Penultimate Bounds}
    I_3 &\ll \frac{(2\pi)^{u}e^{3(u+1) - 3(u+\frac{1}{2})\log (u+1)}}{ (X \Delta)^{u}} \int_{-T}^{T} (u^2+\lvert t \rvert^2 )^{u/2-1/4} e^{-\lvert t \rvert \Delta} \,dt. \notag \\
    &\ll \frac{(2\pi)^{u}e^{3(u+1) - 3(u+\frac{1}{2})\log (u+1)}}{ (X \Delta)^{u}} \int_{0}^{T} (u^2+t^2 )^{u/2-1/4} e^{-t \Delta} \,dt.
\end{align}
We consider the inside integral in \eqref{Penultimate Bounds}. We divide this into two cases. When $t \leqslant u$, we have
\begin{align}\label{t<u}
\int_{0}^{T} (u^2+t^2 )^{u/2-1/4} e^{-t \Delta} \,dt \ll 2^{u/2}u^{u-1/2} \int_{0}^{\infty} e^{-t \Delta} \,dt \ll 2^{u/2}u^{u-1/2} \Delta^{-1}.
\end{align} 
On the other hand, when $t>u$,
\begin{align}\label{t>u}
\int_{0}^{T} (u^2+t^2 )^{u/2-1/4} e^{-t \Delta} \,dt \ll 2^{u/2} \int_{0}^{\infty} t^{u-1/2} e^{-t \Delta} \,dt \ll  2^{u/2}\Gamma(u+1/2)\Delta^{-u-1/2}.
\end{align} 
Therefore, applying \eqref{t<u} and \eqref{t>u} in \eqref{Penultimate Bounds} and using Stirling's approximation for $\Gamma(u+1/2)$,  we obtain
\begin{align}
I_3 &\ll \frac{(2\sqrt{2}\pi)^{u}e^{3(u+1) - 3(u+\frac{1}{2})\log (u+1)}}{ (X \Delta)^{u}}  (u^{u-1/2} \Delta^{-1}+ \Gamma(u+1/2)\Delta^{-u-1/2})\notag \\
&\ll \bigg (\frac{2\sqrt{2}\pi e^2}{X\Delta^2} \bigg)^{u}( u^{-2u-3/2} \Delta^{-1/2}). \label{I3 Bound}
\end{align}

\subsection*{Estimating the horizontal integrals $I_2$ and $I_4$.} We consider $I_4$ only. The treatment for $I_2$ is similar. Dividing the contour into two parts, we write
\[ I_4 =  \int_{-u+iT}^{-3/2+iT}+ \int_{-3/2+iT}^{c+iT} = I_{4,1}+I_{4,2},
\]
say. We work with $I_{4,1}$ first. Applying \eqref{Stirling Formula}, we have
\begin{align}\label{I2 integral}
I_{4,1} = \int_{-u+iT}^{-3/2+iT} \Gamma(s) \zeta(s+1) \frac{1}{\zeta^2(s)} y^s \,ds &\ll \int_{-u+iT}^{-3/2+iT} \frac{(2 \pi)^{2\sigma}\lvert s \rvert^{-\sigma-1/2} e^{-\pi T/2}}{\zeta^3(s)}y^s \,ds \notag \\
&\ll e^{-\pi T/2} \int_{-u+iT}^{-3/2+iT} \frac{(2\pi)^{2\sigma}(\sigma^2+T^2 )^{-\sigma/2-1/4}}{ \zeta^3(s)}y^s \,ds.
\end{align}
Proceeding as in the case of $I_3$, we apply the functional equation for $\zeta(s)$ in \eqref{I2 integral} to obtain
\begin{align*}
I_{4,1} &\ll e^{-T\Delta} \int_{-u+iT}^{-3/2+iT} \frac{(2\pi)^{2\sigma} (X\Delta)^\sigma (\sigma^2+T^2 )^{-\sigma/2-1/4}}{2^{3\sigma} \pi^{3(\sigma-1)}\sin^3(\pi s/2)\Gamma^3(1-s)} \,ds.
\end{align*}
In our range of integration,
\begin{align*}
    \frac{1}{\Gamma(1-s)} \ll e^{1-\sigma - (\frac{1}{2}-\sigma)\log (1-\sigma) + \frac{1}{2}\pi T} \quad \textnormal{as well as} \quad \frac{1}{\sin(\pi s/2)} \ll e^{- \frac{1}{2}\pi T} .
\end{align*}
Therefore, we arrive at
\begin{align}
I_{4,1} &\ll e^{-T\Delta} \int_{-u}^{-3/2} (2\pi)^{-\sigma}(X\Delta)^\sigma (\sigma^2+T^2 )^{-\sigma/2-1/4}e^{3(1-\sigma) - 3(\frac{1}{2}-\sigma)\log (1-\sigma)} \,d\sigma \notag \\
&\ll e^{-T\Delta} \int_{3/2}^{u}  \bigg (\frac{2\pi e^3}{X\Delta} \bigg)^{\sigma} \frac{(\sigma^2+T^2 )^{\sigma/2-1/4}}{ (\sigma+1)^{3(\sigma+\frac{1}{2})}} \,d\sigma. \label{Starting Simplifications 3}
\end{align}
We write \eqref{Starting Simplifications 3} as
\[
 e^{-T\Delta}\int_{3/2}^{u}  \bigg (\frac{2\pi e^3}{X\Delta} \bigg)^{\sigma} \frac{(\sigma^2+T^2 )^{\sigma/2-1/4}}{ (\sigma+1)^{3(\sigma+\frac{1}{2})}} \,d\sigma = e^{-T\Delta}\int_{3/2}^{T}+ e^{-T\Delta}\int_{T}^{u} = J_1 +J_2,
\]
say. Consider $J_1$ first. We have
\[J_1 \ll  e^{-T\Delta}\int_{3/2}^{T}  \bigg (\frac{2\sqrt{2}\pi e^3 T}{\sigma^3} \bigg)^{\sigma}  \,d\sigma.
\]
Let $f$ be defined as
\[
f(\sigma) = \bigg (\frac{2\sqrt{2}\pi e^3 T}{\sigma^3} \bigg)^{\sigma},
\]
and suppose $g(\sigma) = \log f(\sigma)$. The maximum of $g(\sigma)$ is attained at $\sigma_{0} = (2\sqrt{2}\pi T )^{1/3}$. Hence, the maximum value attained by $f$ in our range of integration is $\exp(3(2\sqrt{2}\pi T )^{1/3})$. Therefore, we have
\begin{align}\label{J1 Estimate}
J_1 \ll e^{-T\Delta} e^{3(2\sqrt{2}\pi T )^{1/3}} T \ll e^{-T\Delta/2}.
\end{align}

For $J_2$, we obtain
\begin{align}
e^{-T\Delta}\int_{T}^{u}  \bigg (\frac{2\pi e^3}{X\Delta} \bigg)^{\sigma} \frac{(\sigma^2+T^2 )^{\sigma/2-1/4}}{ (\sigma+1)^{3(\sigma+\frac{1}{2})}} \,d\sigma \ll  e^{-T\Delta}\int_{T}^{u}  \bigg (\frac{2\sqrt{2}\pi e^3}{\sigma^2} \bigg)^{\sigma}  \,d\sigma \ll  e^{-T\Delta/2}. \label{J2 Estimate}
\end{align}

Now, we consider $I_{4,2}$, which is given by
\begin{align*}
    I_{4,2} = \int_{-3/2+iT}^{c+iT} \Gamma(s) \zeta(s+1) \frac{1}{\zeta^2(s)} y^s \,ds.
\end{align*}
In our range of integration,
\begin{align}\label{Gamma bounds and Zeta bounds}
\Gamma(s) \ll T^{\sigma-1/2}e^{-\pi T/2}, \quad \textrm{and} \quad  \zeta(s+1) \ll T^\eta,
\end{align}
for some fixed constant $\eta>0$. We write
\[
I_{4,2} = \bigg(\int_{-3/2+iT}^{-1+iT} +\int_{-1+iT}^{c+iT}\bigg) \Gamma(s) \zeta(s+1) \frac{1}{\zeta^2(s)} y^s \,ds.
\]
For the first part, we apply the functional equation for $\zeta(s)$ to obtain
\begin{align}
\int_{-3/2+iT}^{-1+iT} \Gamma(s) \zeta(s+1) \frac{1}{\zeta^2(s)} y^s \,ds &\ll \frac{T^{\eta-1/2}}{e^{T\Delta}}\int_{-3/2}^{-1}
\bigg(\frac{X\Delta T}{2\pi e}\bigg)^{\sigma}\frac{1}{(1-\sigma)^{1/2-\sigma}} \,d\sigma  \notag\\ 
&\ll \frac{T^{\eta-1/2}}{e^{ T\Delta}}\int_{1}^{3/2}
\bigg(\frac{2\pi e}{X\Delta T}\bigg)^{\sigma}\frac{1}{(1+\sigma)^{1/2+\sigma}} \,d\sigma \ll e^{- T\Delta /2}. \label{I_42 first estimate}
\end{align}
For the second part, using Lemma \ref{lem:RamachandraSankaranarayanan}, for $-1 \leqslant\operatorname{Re}(s)=\sigma \leqslant 2$ and any $\epsilon>0$, for sufficiently large $V>0$, there exists some $V \leqslant T_{*} \leqslant 2V$ such that
\begin{align}\label{Inverse Zeta Bounds}
\frac{1}{\zeta^2(\sigma+iT_{*})}\ll T_{*}^{\epsilon}.  
\end{align}
Therefore, using the bounds \eqref{Gamma bounds and Zeta bounds} and \eqref{Inverse Zeta Bounds} and choosing $T=T_{*}$, we arrive at
\begin{align}
\int_{-1+iT}^{c+iT} \Gamma(s) \zeta(s+1) \frac{1}{\zeta^2(s)} y^s \,ds &\ll  \frac{T^{-1/2+\eta+\epsilon}}{e^{T\Delta}}\int_{-1}^{c}
(X\Delta T)^{\sigma} \,d\sigma \ll  e^{- T\Delta /2}. \label{I_42 second estimate}
\end{align}
Finally combining the estimates \eqref{J1 Estimate}, \eqref{J2 Estimate}, \eqref{I_42 first estimate} and \eqref{I_42 second estimate}, we conclude that $I_4 \ll e^{- T\Delta /2}.$
\subsection*{Estimating the horizontal integrals $I_1$ and $I_5$.} Writing
\[
I_1 = \int_{c+iT}^{c+i\infty} \Gamma(s) \zeta(s+1) \frac{1}{\zeta^2(s)} y^s \,ds,
\]
we note that in our range of integration,
\begin{align*}
    \Gamma(s) &\ll e^{-\pi T/2}, \quad \zeta(s+1) \ll 1, \quad \zeta^{-2}(s) \ll 1, \quad \textrm{and} \quad \lvert y \rvert^s \leqslant(X\Delta)^{\sigma}e^{-T(\Delta-\pi/2)}. 
\end{align*}
Combining these estimates, we obtain $I_1 \ll e^{- T\Delta /2}.$ A similar argument holds true for $I_5$.
\subsection*{Final steps} Combining the estimates from the integrals and letting $N \to \infty$, we obtain
\begin{align}
\Phi_{\mu * \mu}(\rho \e(\theta)) &= 4 \log \frac{X}{1-2\pi i X\theta} - 8 \log(2 \pi) + 72 \frac{1-2\pi i X\theta}{X}\nonumber \\
& \quad +\sum_{|\operatorname{Im}(\rho)| < T_\nu}  f(X,\theta,\rho)  +\sum_{n=1}^{\infty} g(X,\theta,n) +O(e^{- T_\nu\Delta /2}),  \end{align}
where $\nu \leqslant T_\nu \leqslant \nu+1$, $\nu \in \mathbb{N}$, the sum over $\rho$ runs along all the non-trivial zeros of $\zeta$, and $f,g$ are as defined in the statement of the theorem. Finally, we let $ \nu \to \infty$ and choose our sequence of $T_{\nu}$ appropriately  such that \eqref{Inverse Zeta Bounds} is satisfied. We conclude that  
\begin{align*} 
    \Phi_{\mu * \mu}(\rho \e(\theta)) &= 4 \log \frac{X}{1-2\pi i X\theta} - 8 \log(2 \pi) + 72 \frac{1-2\pi i X\theta}{X}\nonumber \\
    & \quad + \lim_{\nu \to \infty} \sum_{|\operatorname{Im}(\rho)| < T_\nu}  f(X,\theta,\rho)  +\sum_{n=1}^{\infty} g(X,\theta,n),
\end{align*}
thereby finishing the last step of the proof.
\qed

\section{Acknowlegements}
We thank Kunjakanan Nath for many comments and discussions on the subject of the paper. We also thank Dorian Goldfeld for useful suggestions. NR wishes to acknowledge support from Vikram Kilambi and from Christopher Pernin.

%%%%%%%%%%%%%%%%%%%%%%%%%%%%%%%%%%%%%%%%%%%%%%%%%%%%%%%%%%%%%%%%%%%%%%%%%%%%%%%%%%%%%%%%%%%%%%%%%
\bibliographystyle{plain}

\end{document}